\documentclass[a4paper]{amsart} 
\usepackage{amsmath}
\usepackage{amssymb}
\usepackage{verbatim}
\usepackage{amsthm}
\usepackage{mathrsfs}
\usepackage{graphicx}
\usepackage{mathtools}
\usepackage[colorlinks,pdfdisplaydoctitle,linkcolor=blue,citecolor=red]{hyperref}
\usepackage{caption}
\usepackage{subcaption}
\usepackage{url}
\usepackage{epstopdf}
\usepackage{pgf,tikz}
\usepackage{todonotes}
\usepackage{bbm}
\usepackage{extarrows}
\usetikzlibrary{arrows}
\usepackage{algorithm}
 \usepackage{algorithmicx}
\usepackage{amssymb,amsthm}    
\usepackage{graphicx}          
\usepackage{amsmath}         
\usepackage{caption}
\usepackage{subcaption}
\usepackage{dsfont}
\usepackage{listings}
\usepackage{bbm}
\usepackage{color}
\usepackage{ dsfont }
\usepackage{enumerate}
 \usepackage{relsize}
\usepackage{xfrac}

\renewcommand{\Delta}{\triangle}

\newcommand{\Cloc}{\widetilde{C}}

\definecolor{darkblue}{rgb}{0,0,0.7}

\definecolor{darkgreen}{rgb}{0.01,0.75,0.24}

\def \Ee[#1]{\mathcal{E}^{\text{{#1}}}}

\def\pa[#1,#2]{\frac{\partial {#1}}{\partial {#2}} }

\def\idom[#1,#2,#3]{\int_{#1}\hspace{1pt} {#2} \hspace{1pt} \text{d}{#3}}
\def\res[#1,#2]{\left.{#1}\right|_{#2}}

\def\var[#1,#2]{\langle \delta \mathcal{E}^{\text{{#1}}}({#2}),v\rangle}
\def\vars[#1,#2,#3]{\langle \delta^2\mathcal{E}^{\text{{#1}}}({#2})v,{#3}\rangle}
\def\vard[#1,#2,#3,#4]{\langle \delta\mathcal{E}^{\text{{#1}}}({#2})-\delta\mathcal{E}^{\text{{#3}}}({#4}),v\rangle}

\def\E{\mathbb{E}}

\newcommand{\ubar}{ \bar{u}}

\newcommand{\reals}{\mathcal{R}}

\newcommand{\dist}{\mathrm{d}}

\newcommand{\ytilde}{\tilde{y}}
\newcommand{\Gtilde}{\widetilde{G}}
\newcommand{\Gbar}{\overline{G}}

\newcommand{\etatilde}{\tilde{\eta}}

\def\ba#1\ena{\begin{align}#1\end{align}}
\def\ban#1\enan{\begin{align*}#1\end{align*}}

\theoremstyle{plain}
\newtheorem{thm}{Theorem}[section]

\newtheorem{lem}[thm]{Lemma}

\newtheorem{aspt}[thm]{Assumption}

\numberwithin{equation}{section}

\begin{document}

\title[Localization in  Ensemble Kalman Inversion]{Localization in  Ensemble Kalman Inversion}

\author[X. T. Tong] {Xin T. Tong}
\address{Department of Mathematics, National University of Singapore, 119077, Singapore}
\email{mattxin@nus.edu.sg}

\author[ M. Morzfeld] {Matthias Morzfeld}
\address{Institute of Geophysics and Planetary Physics, Scripps Institution of Oceanography, University of California, San Diego}
\email{matti@ucsd.edu}


\begin{abstract}
{Ensemble Kalman inversion (EKI) is a technique for the numerical solution of inverse problems. 
A great advantage of the EKI's ensemble approach is that derivatives are not required in its implementation. 
But theoretically speaking, EKI's ensemble size needs to surpass the dimension of the problem.
This is because of EKI's ``subspace property'',
i.e., that the EKI solution is a linear combination of the initial ensemble it starts off with.
We show that the ensemble can break out of this initial subspace when ``localization'' is applied.
In essence, localization enforces an assumed correlation structure onto the problem,
and is heavily used in ensemble Kalman filtering and data assimilation.
We describe and analyze how to apply localization to the EKI,
and how localization helps the EKI ensemble break out of the initial subspace.
Specifically, we show that the localized EKI (LEKI) ensemble will collapse to a single point (as intended)
and that  the LEKI ensemble mean will converge to the global optimum at a sublinear rate.
Under strict assumptions on the localization procedure and observation process,
we further show that the data misfit decays uniformly.
We illustrate our ideas and theoretical developments with numerical examples with simplified toy problems, 
a Lorenz model, and an inversion of electromagnetic data,
where some of our mathematical assumptions may only be approximately valid.}
\end{abstract}

\maketitle
%

\section{introduction}
\label{sec:intro}

Many problems in science and engineering require parameter estimation of a mathematical model from data, see e.g. \cite{Parker94, ABT13, Tarantola05}. 
Such an inverse problem is typically based on the equation
\begin{equation}
\label{eq:inv}
y=G(u) + \eta,\quad \eta\sim \mathcal{N}(0,I_{d}),
\end{equation}
where $y \in \reals^{d_y}$ are the data, $u\in \reals^{d_u}$ are the unknown model parameters,
 $G:\reals^{d_u}\mapsto \reals^{d_y}$ is the model (often a discretization of a differential equation),
 and where the random variable $\eta$ represents observation errors between model outputs and data.
 Throughout, we assume that these observation errors are Gaussian.
 The overall goal of solving an inverse problem is to find $u$, given $G$ and $y$,
 along with statistical assumptions about the errors $\eta$.
 
Ensemble Kalman inversion (EKI) is a computational strategy for 
solving inverse problems \eqref{eq:inv} \cite{iglesias2013ensemble, schillings2018convergence}.
Its formulation is inspired by the ensemble Kalman filter (EnKF) \cite{E94}, which is an ensemble-based algorithm originally designed for high dimensional data assimilation problems (see also, e.g., \cite{sr:evensen03, sr:evensen04, jdw:Evensen2006}).
The EKI works by iteratively updating an ensemble 
of candidate solutions $\{u^{j}(n)\in \reals^{d_u}\}_{j=1}^J$ from iteration index $n$ to $n+1$
(see Section \ref{sec:EKI} for details).
Very often, the EKI ensemble collapses to a single point after several iterations, and this point can be seen as a minimizer of 
the loss function
\begin{equation}
\label{eq:loss}
l(u)=\|G(u)-y\|^2.
\end{equation}
We note that EKI is an optimization algorithm that does not require derivatives
(as opposed to gradient descent, Newton, Gauss-Newton or Quasi-Newton methods).
There are many recent works that carefully describe the mathematics of EKI \cite{schillings2018convergence, ding2021ensemble}, 
that suggest improvements \cite{chada2018parameterizations, chada2019convergence, chada2019incorporation, iglesias2021adaptive},
and that explain how EKI can be used in machine learning \cite{kovachki2019ensemble}.
An extension of EKI, such that the EKI ensemble is distributed according to a Bayesian posterior distribution, is discussed in \cite{ding2021sample} and is called the ensemble Kalman sampler.
The EKI can further be extended to include a Tikhonov  regularization term in the loss function,
which avoids overfitting the data \cite{chada2020tikhonov};
this version of the EKI is called TEKI (with ``T'' standing for Tikhonov regularization).
Algorithms similar to EKI can also be derived as an ensemble randomized maximum likelihood solver \cite{gu2007iterative, raanes2019revising},
and  can be applied to history matching problems,
where they are better known as iterative ensemble Kalman smoothers \cite{chen2010cross, chen2012ensemble, bocquet2014iterative}.

One fundamental issue of EKI, which is the main point we address in this paper,
is its \emph{subspace property} \cite{iglesias2013ensemble, schillings2018convergence}. 
Put simply, the subspace property implies that all ensemble members  $\{u^{j}(n)\in \reals^{d_u}\}_{j=1}^J$ 
are confined to the linear subspace $S_0$ spanned by the initial ensemble $\{u^{j}(0)\in \reals^{d_u}\}_{j=1}^J$. 
In practice, this means that the initial subspace needs to be rich enough to contain the global, or at least a ``useful,''
minimizer of the loss function \eqref{eq:loss}. 
One way to achieve this is to use an ensemble size $J\geq d_u+1$ so
that the dimension of $S_0$ is equal to $d_u$.
Such a linear scaling of required ensemble size with number of unknown parameters
is too expensive for large-scale problems.
The rest of this paper describes 
one computational strategy that can overcome this issue,
thus turning EKI into a possibly efficient numerical technique for solving large-scale inverse problems
in which derivatives are hard to come by.

Our theory is inspired by the literature on ensemble data assimilation (DA)
where similar issues occur.
Ensemble DA, in particular EnKF \cite{E94} or ensemble-variational schemes \cite{Lorenc2003},
are routinely used in operational numerical weather prediction to estimate millions of unknowns,
but with an ensemble size of a few hundred. 
This incredible efficiency (small ensemble size) is achieved by a technique termed \emph{localization}
\cite{HoutekamerMitchell1998, Houtekamer2001, Hamill2001, Ott_etal2004}.
In essence, localization enforces on the sample covariance from an ensemble the assumption 
that the covariance between two locations will be small for sufficiently large separations.
Localization is typically implemented by artificially reducing, or truncating, correlations
within the ensemble that are deemed ``spurious'' (based on above assumptions about an expected, spatial decay of correlations). 
Localization, however, also boosts the rank of the ensemble covariance
and is, more generally, a technique for reducing sampling error \cite{A07}. 
The latter is particularly inspiring for us, because it can help breaking EKI's subspace property. 

The remainder of this paper describes the mathematics of leveraging ideas akin to localization in EKI.
Indeed, localization in ensemble DA and in inverse methods similar to EKI \cite{bocquet2014iterative, chen2017localization} is a heuristic process 
with nearly no mathematical justification.
This paper intends to bridge this gap between practice and theory.
Specifically, we present mathematical analyses for the following key issues:
\begin{enumerate}[(i)]
\item
a localized EKI (LEKI) ensemble collapses to a single point at a specified rate;
\item 
a LEKI ensemble collapses onto the global minimizer of the loss function uniformly over all components and at a sublinear rate.
\end{enumerate}
Such results are important for the applicability of EKI to large scale inverse problems,
because the LEKI ensemble breaks out of its initial subspace,
thus doing away with linear scaling requirements of the ensemble size with dimension.
In passing, we note that our results equally apply to TEKI (Tikhonov regularized EKI),
because TEKI can be implemented by simply extending the observations and unknown parameters, as shown in \cite{chada2020tikhonov}.
We briefly discuss how to achieve this in Section \ref{sec:TEKI} and demonstrate our theory for TEKI at a few  numerical examples in Section \ref{sec:num}.

The remainder of this paper is organized as follows. 
In Section \ref{sec:back}, we review the background of EKI and its Tikhonov regulariation.
In Section \ref{sec:locsch}, we introduce two localization schemes and set up the mathematical problems.
We study the ensemble collapse question in Section \ref{sec:EKI}. Theorems \ref{thm:max} and \ref{thm:min} show that $\Cloc^{uu}(t)$ converges to zero like $O(\frac1t)$ for both its maximum and minimum eigenvalues. Section \ref{sec:opt} explains when LEKI converges to the optimal solution. In particular, Theorem \ref{thm:globalloss} provides sublinear convergence in the loss function $l$ and Theorem \ref{thm:locopt} gives uniform sublinear convergence for each individual local misfit $l_i(u)$. Section \ref{sec:num} demonstrates the effectiveness of LEKI on toy linear and nonlinear problems, the Lorenz' 96 problem and the inversion of DC resistivity data. 

\section{Background of  EKI}
\label{sec:back}
\subsection{Ensemble Kalman Inversion and its continuous-time limit}
\label{sec:EKI}
Ensemble Kalman inversion (EKI) relies on an ensemble of candidate solutions $\{u^{j}(n)\in \reals^{d_u}\}_{j=1}^J$.  We use $j$ to index the ensemble member, $n$ to index the iteration number, and $J$ for the ensemble size. 
The EKI ensemble is updated in an iterative process as follows.
One first computes the sample averages
\[
\bar{u}(n) = \frac{1}{J}\sum^{J}_{j=1}u^{j}(n), \quad \Gbar(n) = \frac{1}{J}\sum^{J}_{j=1}G(u^{j}(n)),
\]
and the sample (cross) covariance matrices
\begin{subequations}
\begin{align*}
C^{uu}(n) & = \frac{1}{J-1}\sum^{J}_{j=1} \bigl(u^{j}(n) - \ubar(n)\bigr)\otimes 
 \bigl(u^{j}(n) - \ubar(n)\bigr),\\ 
C^{up}(n) &= \frac{1}{J-1}\sum^{J}_{j=1} \bigl(u^{j}(n) - \ubar(n)\bigr)\otimes \bigl(G(u^{j}(n)) - \Gbar(n)\bigr),\\ 
C^{pp}(n) & = \frac{1}{J-1}\sum^{J}_{j=1}  \bigl(G(u^{j}(n)) - \Gbar(n)\bigr)\otimes  \bigl(G(u^{j}(n)) - \Gbar(n)\bigr),
\end{align*}
\end{subequations}
for the current iteration step ($n$).
Here and below, we use $a\otimes b$ to denote $a b^\top$,
where the superscript $\top$ is a transpose.
The ``classical'' EKI \cite{iglesias2013ensemble, schillings2018convergence} 
update of the ensemble is then given by:
\begin{equation}
\label{eq:clEKI}
u^{j}(n+1) = u^{j}(n) + C^{up}(n) \big(C^{pp}(n) +  I\big)^{-1}\big(y - G(u^{j}(n))\big).
\end{equation}
The update mechanism reveals the \emph{subspace property} of EKI.
We note from \eqref{eq:clEKI} that if $v$ is a vector perpendicular to all $u^{j}(n)$
($v^\top u^{j}(n)=0$ for all $j$), then it remains perpendicular
to the ensemble at the next and every iteration. 
It follows that the EKI ensemble,
at every stage of the iteration, 
is confined to the subspace spanned by the initial ensemble. 

From a theoretical Kalman filtering perspective \cite{kalman1960new}, \eqref{eq:clEKI} also lacks the controllability so the ensemble may collapse too fast. 
One standard method to avoid collapse and improve controllability is additive inflation
\cite{Anderson1999}. 
Here, we consider inflation via a set of vectors $\xi^{j}(n)$, $j=1,\dots,J$, such that
$\sum_{j=1}^J \xi^{j}(n)=0,$ and
\[
\frac1{J-1}\sum_{j=1}^J\left(\xi^{j}(n)\otimes (u^j(n)-\ubar(n))+(u^j(n)-\ubar(n))\otimes \xi^{j}(n)\right)=\Sigma(n),
\]
where $\Sigma$ is a matrix whose diagonal elements are equal to $1$. 
One can generate this set of vectors by applying a component-wise whitening transformation to the ensemble and let
\[
	\xi^j(n)=\frac12 D(n)^{-1}(u^j(n)-\ubar(n)),
\]
where $D(n)$ is the diagonal part of $C^{uu}(n)$.
Adding these perturbations to the standard EKI gives the inflated EKI update
\begin{equation}
\label{eq:disupdate}
u^{j}(n+1) = u^{j}(n) + C^{up}(n) \big(C^{pp}(n) +  I\big)^{-1}\big(y - G(u^{j}(n))\big)+\lambda_n \xi^j(n),
\end{equation}
where $\lambda_n$ controls the strength of the inflation. 
It is easy to see that $u^{j}(n)$ is no longer confined to the subspace $S_0$, 
but rather the ensemble is in a larger subspace that includes the perturbation vectors.
Additive inflation in itself, however, may not lead to a subspace rich enough to
contain the global minimizer or even a ``useful'' solution if $d_u$ is large.

Following \cite{schillings2017analysis}, we focus on the continuous-time limit of the EKI update \eqref{eq:disupdate}, because this will simplify our analysis.
The continuous-time limit involves first replacing $n$ with $nh$ and rescaling the adjustment terms
in the update equation:
\[
u^{j}((n+1)h)- u^{j}(nh) = C^{up}(nh) \big(C^{pp}(nh) + h^{-1} I\big)^{-1}\big(y - G(u^{j}(nh))\big)+h\lambda_{nh}\xi^j.
\]
Dividing both sides by $h$ and taking the $h\to 0$, one obtains 
the EKI continuous-time limit in form of an ordinary differential equation (ODE):
\begin{equation}
\label{eq:ctEKI}
\frac{d}{dt}u^{j}(t)=-C^{up}(t) (G(u^{j}(t))-y)+\lambda_t \xi^{j}(t).  
\end{equation}
The rigorous proof showing that $\eqref{eq:ctEKI}$ is the  $h\to 0$ limit  can be found in \cite{blomker2018strongly, blomker2021continuous}.
Here, $\xi^j(t)$ are such that $\sum_j\xi^j(t)=0$ and 
\[
\frac1{J-1}\sum_{j=1}^J\left(\xi^{j}(t)\otimes (u^j(t)-\ubar(t))+(u^j(t)-\ubar(t))\otimes \xi^{j}(t)\right)=\Sigma(t)
\]
for some matrix $\Sigma(t)$ with diagonal terms being $1$. 
And throughout this paper, we set $\lambda_t=\sigma/(t+1)^2$, 
where $\sigma\geq 0$. 
The reasons for this choice of $\lambda_t$ will be explained at the end of Section \ref{sec:obs}. 
Note that one obtains the standard continuous time limit of EKI for $\sigma=0$
(no inflation). 

Before we move on, it can be useful to pause and provide an intuitive explanation for why EKI may not converge  to the correct solution if the ensemble size $J$ is small (less than $d_u$). Consider a simple case where $G(u)=u$, because $\sum_{j=1}^J \xi^j(t)=0$,  the ensemble mean of \eqref{eq:ctEKI} follows 
\begin{equation}
\label{eq:ctEKImean}
\frac{d}{dt}\bar{u}(t)=-C^{uu}(t) (\bar{u}(t)-y).
\end{equation}
If $\ubar(t)$ converges to a fixed  point $u^*$, then $C^{uu}(t)(u^*-y)=0$. When the ensemble size is large and $C^{uu}(t)$ is full rank, then we can conclude that $u^*=y$, which is the optimal solution. But with a limited ensemble size, $C^{uu}(t)$ is of low rank, and hence there is no guarantee that $u^*=y$.

\subsection{Tikhonov regularized ensemble Kalman inversion (TEKI)}
\label{sec:TEKI}
The classical formulation of EKI does not account for regularization,
which can lead to overfitting  \cite{engl1996regularization, hanke1997regularizing}. 
Here we consider a Tikhonov regularized loss function 
\begin{equation}
\label{eq:rloss}
l_{\text{Tik}} (u)=\|G(u)-y\|^2+\|u\|^2_{C_0^{-1}},
\end{equation}
where we use the shorthand notation $\|u\|^2_A:=u^\top Au$ (similar to a Mahalanobis norm)
and where $C_0$ is a positive definite matrix.
From the Bayesian perspective, Tikhonov regularization is equivalent to choosing a Gaussian prior $u\sim \mathcal{N}(0, C_0)$, and the minimum of \eqref{eq:rloss} is the maximum a posteriori estimator \cite{stuart2010inverse}. 

One can incorporate a Tikhonov regularization into EKI (TEKI)  by extending the data and parameter vectors, 
as first documented in \cite{chada2020tikhonov}.
Specifically, we define the extended observations, model and observation errors as
\[
\ytilde=\begin{bmatrix} 0\\ y\end{bmatrix}\quad \Gtilde(u)=\begin{bmatrix} C_0^{-1/2}u\\ G(u) \end{bmatrix},\quad \etatilde=\begin{bmatrix} -C_0^{-1/2}u\\\eta\end{bmatrix}.
\]
With these extensions, we obtain an extended loss function $\tilde{l}$ of the form \eqref{eq:loss},
which is in fact equal to the Tikhonov regularized loss function \eqref{eq:rloss}:
\[
\tilde{l}(u)=\|\Gtilde(u)-\ytilde\|^2=\|G(u)-y\|^2+\|u\|^2_{C^{-1}_0}=l_{\text{Tik}}(u).
\]
Thus, upon the above transformation\slash extension, 
EKI and TEKI are equivalent and we will focus our discussion and analysis on the EKI.
Some numerical experiments will demonstrate how our analysis equally applies to TEKI. 
As a final remark, one can also consider integrating regularization through the randomized maximal likelihood approach
\cite{raanes2019revising}.

\subsection{Notation}
To facilitate our discussion, we adopt the following notations. 
Given a matrix $A$, we use $A_{i,j}$ or $[A]_{i,j}$ to denote its $i,j$-th entry.
With two symmetric matrices $A$ and $B$, we write $A\preceq B$ if $B-A$ is positive semidefinite.
We write the Schur product as $[A\circ B]_{i,j}=A_{i,j}B_{i,j}$. The well known Schur product theorem indicates that if $A\preceq B$, then for any other positive semi-definite matrix $C$, $A\circ C\preceq B\circ C$.

With a real symmetric matrix $C$, we use $\|C\|$ to denote the $l_2$ operator norm
(the largest eigenvalue of $C$). We use $\lambda_{\min}(C)$ to denote the smallest eigenvalue of $C$.
 If $C$ is positive semidefinite, we define the maximum entry as 
\[
\|C\|_{\max}=\max_{i,j} |C_{i,j}|=\max_{i,i} C_{i,i}
\]
where the second identity comes from the positive  semidefiniteness of $C$. We also define the following norm 
\[
\|C\|_{1}=\max_{i}\sum_{j=1}^{d_u}|C_{i,j}|
\]
which is also the $l_\infty$ operator norm. See Lemma \ref{lem:norm} for some relationships among these norms. 

Our analyses will be focusing on LEKI's performance when $t$ and $d_u$ are large. 
In this context, it is convenient to treat other independent parameters as constants. 
This leads to the standard big $O$ and ``$\lesssim$" notation.  In particular, we say $f(t, d_u,d_y)$ is $O(g(t,d_u,d_y))$ or $f(t, d_u,d_y)\lesssim g(t,d_u,d_y)$, if there is a constant $C$ independent of $t,d_u, d_y$ such that 
\[
f(t, d_u,d_y)\leq C g(t,d_u,d_y).
\]
This also implies all other constants in the assumptions are assumed to be independent of $t,d_u$ and $d_y$. 

\section{Localization of the EKI}
\label{sec:locsch}
We will now show how ideas akin to localization in ensemble DA
can be used within the EKI and, in particular, how localization can enrich the ensemble subspace.
Localization has its roots in data assimilation problems with an inherently spatial interpretation (hence the name). More broadly, however, localization is an effective means for reducing sampling error that arises due to a small ensemble size~\cite{A07}.

\subsection{Localization in ensemble DA}
\label{sec:EnKFLoc}
Briefly, localization in ensemble DA for spatial problems is as follows.
Each model component, $u_i$, is associated with a spatial location,
and two model components $u_i$ and $u_j$ are separated by a distance $\dist(i,j)$.
It is \emph{assumed} that the covariance between $u_i$ and $u_j$ decays with distance.
Localization then amounts to enforcing this covariance structure onto the 
ensemble covariance $C^{uu}$.
Localization is often implemented via Schur products:
\begin{equation}
\label{eqn:uuloc}
\Cloc^{uu}_{i,j}=C^{uu}_{i,j}\Psi_{i,j},\quad \Psi_{i,j}=\psi(\dist(i,j)/R_l).
\end{equation}
Here, $R_l$ is a decorrelation length scale and $\psi$ is the localization function,
e.g., the Gaspari Cohn function which tapers to zero \cite{Gaspari1999}.
The localized ensemble covariance $\Cloc^{uu}$ then
 replaces $C^{uu}$ in the ensemble DA.
By the Schur product theorem,
a suitably chosen localization increases the rank of $C^{uu}$ to be larger than the ensemble size.
For this reason, the ensemble after an update \eqref{eqn:uuloc} 
is \emph{not} a linear combination of the ensemble at time $n$.

\subsection{Localization schemes for EKI}
To implement localization within an EKI, 
we need to define a localization of the \emph{cross} covariance $C^{up}$ -- 
the covariance between model parameters and model outputs.
This is a non-trivial problem and existing works are mostly heuristic guidelines
(see, e.g., \cite{chen2010cross, chen2017localization}). 
Below, we translate such guidelines into precise mathematical formulas. 

\subsubsection{Linear and linearized  localization}
We start with a localization scheme that relies on two assumptions:
\begin{enumerate}[(i)]
\item the model is linear, i.e., $G(u)=Hu$ for some matrix $H$
\item the parameter-parameter covariance $C^{uu}$ can be localized as explained in section
\ref{sec:EnKFLoc}.
\end{enumerate}
While the assumptions are restrictive, they represent a good starting point for a more general theory (see below).
Moreover, assumption (ii) is often easy to satisfy in practice where 
correlation structure of the parameters is known, 
e.g., because one can rely on smoothness assumptions or is aware of inherent spatial scales.

Due to the (assumed) linearity of the model, the cross covariance is $C^{up}=C^{uu}H^\top$.
Since we know how to localize $C^{uu}$,
a natural way to localize the cross covariance $C^{up}$ is to set
\begin{equation}
\label{eq:linloc1}
\Cloc^{up}=\Cloc^{uu} H^\top,
\end{equation}
where $\Cloc^{uu}$ is the localized parameter-parameter covariance.
Indeed, this localization scheme is often used in the EnKF literature
(see, e.g., \cite{houtekamer2001sequential,houtekamer2005atmospheric}),
or used to create guidelines for more sophisticated localization schemes.
By the Schur theorem, localization increases the rank of the matrices
$C^{uu}$ and $C^{up}$.
Thus, the EKI update \eqref{eq:disupdate},
and its continuous-time limit \eqref{eq:ctEKI},
no longer generates a linear combination of the ensemble at time $n$.
Due to localization, the EKI ensemble can break out of the subspace
spanned by the initial ensemble.

Relaxing the assumption of a linear model,
the above scheme can be applied to nonlinear problems by (approximately) linearizing the model.
Specifically, suppose $H(t)$ is an approximation of the Jacobian $\nabla G(\ubar(t))$,
obtained, e.g., via an ensemble based sensitivity analysis or adjoint model \cite{emerick2011combining, bocquet2014iterative}.
Then a nonlinear, but linearized localization scheme is
\begin{equation}
\label{eq:linloc}
\Cloc^{up}(t)=\Cloc^{uu}(t) H(t)^\top.
\end{equation}
Admittedly, finding the Jacobian is computationally difficult and 
defeats the purpose of using LEKI as a derivative free algorithm.
It is of this reason, our subsequent analysis does not require the exact Jacobian, 
but rather an approximation (See Assumption \ref{aspt:uloc} for details). On the other hand, 
the accuracy of LEKI depends on the accuracy of the approximate Jacobian.

%
\subsubsection{Centralized localization}
In many applications, observations are made at a single location (local observations).
Mathematically, this means $G_j(u)=G_j(u_{i(j)})$ for some $i(j)\in \{1,\ldots, d_u\}$. 
Then, naturally, the distance between the $j$-th observation and the $i$-th model component is given by $\dist(i,i(j))$,
 and we can use a typical localization function, such as Gaspari-Cohn,
 to modify the cross covariance \cite{chen2010cross, luo2020automatic}: 
\begin{equation}
\label{eq:cenloc}
\Cloc^{up}_{i,j}=C^{up}_{i,j} \Psi_{i,i(j)}. 
\end{equation}
In more general settings,  $G_j$ is approximately ``local" if it concerns only  state variables near $i(j)$,
that is
\[
G_j(u)=G_j(u_{I_j}),\quad I_j=\{i: d(i,i(j))<l\},
\]
for some $l$ as the radius of the neighborhood. 
In particular, the $i$-th component of the Jacobian $\nabla G_{i,j}(u)$ is zero if $i\notin I_j$,
so that the Jacobian $\nabla G$ is a sparse matrix (which is nearly always the case in practice).
We call observations of this type \emph{centralized}
and we can apply the localization rule as above to centralized observations. 
The accuracy of this localization depends on the degree of centralization of $I_j$ around $i(j)$. 

Recent studies also suggest localization schemes of the form $\Cloc^{up}_{i,j}=C^{up}_{i,j} \Phi_{i,j}$ 
may achieve better recovery if $\Phi_{i,j}$ is obtained through proper correlation testing \cite{luo2018correlation,luo2019correlation,luo2020automatic}. 
Our analyses below in principle apply to such localization schemes,
but verifying the assumptions is challenging. 

\subsubsection{Linearized and centralized observations} 
Observations of both types (linearized and centralized) can also be handled easily
by applying the above localization strategies separately to each observation type.
For example, if  $G(u)=[L(u), A(u) ]$, $L(u)=[G_{1}(u_{j_1}),\ldots, G_l(u_{j_l})]$,  $A(u)\approx Hu$ for some $H\in \reals^{d\times (p-l)}$ can be approximated by linear observation, and then 
\[
\Cloc^{up}_{i,j}=\begin{cases}
C^{up}_{i,j} \Psi_{i,j},\quad j\leq l\\
[\Cloc^{uu} H^\top]_{i,j-l},\quad j>l,
\end{cases}
\]
can be used as the localized cross-covariance matrix.

\subsection{Localized EKI and main objectives}
In summary, we implement the above cross covariance localization schemes
to obtain the localized EKI (LEKI):
\begin{equation}
\label{eq:ctlEKI}
\frac{d}{dt}u^{j}(t)=-\Cloc^{up}(t) (G(u^{j}(t))-y)+\lambda_t \xi^{j},
\end{equation}
where $\Cloc^{up}(t)$ is the \emph{localized} ensemble cross-covariance.
This means that localization is in principle straightforward:
simply replace the ensemble covariance by a localized ensemble covariance.

The advantage of LEKI is that the localization increases the rank of $C^{up}$,
so that the ensemble subspace is enriched, 
breaking out of the subspace spanned by the initial ensemble.
To see this, revisit the simple case where $G(u)=u$. The ensemble mean of \eqref{eq:ctlEKI} the localized EKI  follows 
\[
\frac{d}{dt}\bar{u}(t)=-\Cloc^{uu}(t) (\bar{u}(t)-y).
\]
If $\ubar(t)$ converges to a fixed  point $u^*$, then $\Cloc^{uu}(t)(u^*-y)=0$. Due to the localization, $\Cloc^{uu}$ can be of full rank even if the ensemble size is small. This indicates that $u^*=y$, which is the correct solution.

The enrichment of  subspace
means that the ensemble size need not scale with the dimension,
which is critical for practical application of LEKI.
The localization, however, must be done carefully, or else the LEKI looses
the important properties of the EKI.

This paper represents a first step towards understanding how localization 
enables EKI to function in practice,
i.e., with nonlinear models and with a small ensemble size.
Our work is mathematically rigorous and complements the work of practitioners,
who study the effects of localization in the context of specific scientific problems.
As an aside, we hope to spark more interest in localization within the applied mathematical community,
because we demonstrate that localization is a mathematically sound idea for breaking the subspace property,
hence making ensemble-based methods feasible in large-scale problems.
In this context, it is noteworthy that ensemble-based algorithms,
that are not localized, are no longer acceptable within the field of numerical weather prediction \cite{HWAS09}.

The rest of this paper addresses the following three important questions:
\begin{enumerate}[(i)]
\item Will the LEKI ensemble, $\{u^j(t)\}_{j\leq J}$, collapses to a single point? If yes, how fast does it collapse?
\item Will the LEKI ensemble mean, $\ubar^j(t)$, 
converge to the global minimizer of the loss function $l(u)$ if 
the ensemble size is less than the number of unknown parameters ($J\ll d_u$)? 
If yes, what is the rate of convergence?
\item Does the error concerning the $j$-th observation, $l_j(\ubar(t))=|G_j(u)-y_j|^2$, decay uniformly over all $i$? 
If yes, what is the rate of convergence?
\end{enumerate}
Rigorous convergence analysis of EKI addressing questions (i) and  (ii) is known to be difficult. 
Existing work typically assumes that $G$ is linear \cite{schillings2017analysis,schillings2018convergence,chada2020tikhonov},
and\slash or assumes the mean field limit (infinite ensemble size,  $J\to \infty$) \cite{ding2021ensemble}. 
For example, \cite{chada2019convergence} addresses nonlinear problems,
but assumes that the ensemble size is large ($J\geq d_u$). 
Such results are important and necessary  first steps,
but are ultimately of minor practical relevance,
because nearly all relevant models are nonlinear and a large ensemble size is impractical.
Question (iii) is important to understand 
a ``local'' error $l_j(u)=|G_j(u)-y_j|^2$ (note that $l(u) = \sum l_j(u)$). 
Understanding uniform convergence is practically relevant,
but to the best of our knowledge, questions of this type
have been studied only in the context of EnKF, and under restrictive conditions \cite{de2020analysis}.

\section{Ensemble Collapse of LEKI}
\label{sec:collapse}
One practical indication that a LEKI 
has converged onto a solution of the inverse problems
is that the ensemble collapses. 
Thus, a natural first step of our analysis is to show that the LEKI ensemble \eqref{eq:ctlEKI} collapses.

For this purpose, we first note that ensemble average follows the stochastic ordinary differential equation (ODE):
\[
\frac{d}{dt}\ubar(t)=\frac1J\sum_{j=1}^J \frac{d}{dt}u^j(t)=- \Cloc^{up}(t)(\Gbar(t)-y),
\]
where $\Gbar(t)=\frac{1}{J}\sum_{j=1}^J G(u^j(t))$. The ensemble deviation $v^j(t)=u^j(t)-\ubar(t)$ follows the ODE
\[
\frac{d}{dt}v^j(t)=- \Cloc^{up}(t)(G(u^j(t))-\Gbar(t))+\lambda_t \xi^j(t).
\]
So the parameter-parameter covariance $C^{uu}=\frac{1}{J-1}\sum_{j} v^j(t)\otimes v^j(t)$ satisfies the ODE
\begin{align}
\notag
\frac{d}{dt} C^{uu}(t)&=\frac{1}{J-1}\sum_{j=1}^J \left(\Cloc^{up}(t)(G(u^j(t))-\Gbar(t))\otimes v^j(t)+\lambda_t \xi^j(t)\otimes v^j(t)\right)\\
\notag
&\quad+\frac{1}{J-1}\sum_{j=1}^J \left(v^j(t)\otimes \Cloc^{up}(t)(G(u^j(t))-\Gbar(t)) +\lambda_t v^j(t)\otimes \xi^j(t)\right)\\
\label{eqn:CuuODE}
&=-\Cloc^{up}(t) C^{pu}(t)-C^{up}(t)\Cloc^{pu}(t)+\lambda_t\Sigma(t).
\end{align}
Note that $C^{uu}$ is the ensemble covariance of the \emph{localized} ensemble because
we use the localized cross-covariance $\Cloc^{up}$.
Thus, when $C^{uu}$ goes to zero,
the localized ensemble collapses.

Note that the localization increases the rank of the covariance and cross-covariance matrices
and, for that reason, LEKI explores and collapses within a subspace
that is richer than the EKI subspace, spanned by the initial ensemble.
Proving the collapse of LEKI in the enriched subspace is non-trivial
and our proof requires two steps.
We first show in Theorem \ref{thm:max},
that $\max_iC^{uu}_{i,i}(t)$ decays at least like $1/t$ (under observability conditions).
We then prove Theorem \ref{thm:min},
which shows that $\min_iC^{uu}_{ii}(t)$ decays at most like  $1/t$ (under regularity conditions).
Combining both results, we conclude that the maximum and minimum  eigenvalues of $\Cloc^{uu}$ decay like $1/t$,
and, hence, the LEKI ensemble collapses and with appropriate localization  the ensemble covariance is enriched to be full-rank.
 

\subsection{Observability and covariance upper bound}
\label{sec:obs}
In the classical Kalman filter theory, posterior covariance upper bounds can usually be established when the system is observable \cite{kalman1960new}.
We follow these ideas and assume that all parameters are observable,
because, otherwise, the LEKI ensemble may not collapse in unobserved directions.
Based on \eqref{eqn:CuuODE}, we formalize observability by the following assumption.

\begin{aspt}[Observability]
\label{aspt:obs}
The following holds with a constant $c_o>0$ for all $t\geq 0$:
\[
[\Cloc^{up}C^{pu}]_{i^*,i^*}=\sum_{j=1}^{d_y} \Cloc^{up}_{i^*,j}C^{up}_{i^*,j}\geq c_o (C^{uu}_{i^*,i^*})^2,
\]
where $i^*\in \{1,\ldots, d\}$ is the index  such that $C^{uu}_{i^*,i^*}=\| C^{uu}\|_{\max}$.
\end{aspt}
Note when the observation map $G$ is linear, 
a full rank $H$ implies observability (with or without localization).

Next we provide sufficient conditions for  the two localization schemes in Section \ref{sec:locsch}
to satisfy Assumption \ref{aspt:obs}.
\begin{lem}
\label{lem:obsCentral}
With the centralized localization scheme, suppose for each $i\in \{1,\ldots,d_y\}$, there is a  local observation $j$ such that $i(j)=i$ and $G_j(u)$ depends mostly on $u_i$, that is for some $l_1\geq 0$
\[
\partial_{u_i}G_{j}-\sum_{k\neq i}|\partial_{u_k}G_{j}|\geq l_1. 
\] 
Then Assumption \ref{aspt:obs}  holds with $c_o=l_1$. 
\end{lem}
\begin{proof}
Let $i^*$ be the index so that $[C^{uu}]_{i^*,i^*}\geq [C^{uu}]_{j,k}$ for all $j$ and $k$. 
With the centralized localization scheme, $ \Cloc^{up}_{{i^*},j}=C^{uu}_{{i^*},j}\Psi_{i,i(j)}$, so $ \Cloc^{up}_{{i^*},j}C^{up}_{{i^*},j}\geq 0$, and, therefore,
\begin{align}
\label{eq:Lemma42}
\sum_{j=1}^{d_y} \Cloc^{up}_{{i^*},j}C^{up}_{{i^*},j}\geq (C^{up}_{{i^*},j})^2.
\end{align}
Meanwhile, if $i(j)={i^*}$, let $u^{m,n}_s=u^n+s(u^m-u^n)$ for $s\in [0,1]$, 
\begin{align*}
C^{up}_{{i^*},j}&=\frac{1}{J(J-1)}\sum_{m,n} (G_j(u^m)-G_j(u^n))(u^m_{i^*}-u^n_{i^*})\\
&=\frac{1}{J(J-1)}\sum_{m,n} \sum_k\left(\int^1_0 \partial_kG_j(u^{m,n}_s)ds\right)(u^m_k-u^n_k)(u^m_{i^*}-u^n_{i^*})\\
&\geq \frac{1}{J(J-1)}\sum_{m,n} \left(\int^1_0\left( \partial_{i^*}G_j(u^{m,n}_s)-\sum_{k\neq {i^*}}|\partial_kG_j(u^{m,n}_s)|\right)ds\right)(u^m_{i^*}-u^n_{i^*})^2\\
&\geq \frac{1}{J(J-1)}\sum_{m,n} l_1|u^m_{i^*}-u^n_{i^*}|^2\geq l_1 C^{uu}_{{i^*},{i^*}}. 
\end{align*}
Squaring the last inequality and combining with \eqref{eq:Lemma42}
proves the Lemma.
\end{proof}

Note that the assumption is easier to satisfy when the ensemble size ($J$) is large,
but we do not require that $J$ is large, or larger than the overall dimension.

Assumption \ref{aspt:obs} for the 
linearized localization scheme requires additional conditions:
(i) each observation is a spatial shift of another one; 
and (ii) both $H^\top H$ and the localization matrix decay quickly in the off-diagonal direction. 
More formally, we have the following Lemma.
\begin{lem}
\label{lem:obsLinear}
With the linearized localization scheme, suppose the linearization matrix $H$ and localization matrix $\Psi$
are such that 
\begin{enumerate}
\item All diagonal terms of $H^\top H$ takes the same value $H_0$.
\item There is an $0\leq h_0<1$ so that $\sum_{j\neq i}[H^\top H]_{j,i}<h_0 [H^\top H]_{i,i}$ for all $i$.
\item There is an $0\leq \psi_0<1$ so that $\Psi_{i,i}=1,\sum_{j\neq i}\Psi_{i,j}<\psi_0$ for all $i$.
\end{enumerate}
Then Assumption \ref{aspt:obs}  holds with $c_o=H_0 (1- (h_0+\psi_0(1+h_0)))$. 
\end{lem}

\begin{proof}
Let $i$ be  the index so that $[C^{uu}]_{{i^*},{i^*}}$ is maximized.
Then note that,
\begin{align*}
\sum_{j=1}^{d_y} \Cloc^{up}_{{i^*},j}C^{up}_{{i^*},j}&=[\Cloc^{uu}H^\top  H C^{uu}]_{{i^*},{i^*}}\\
&=\sum_{k,j=1}^{d_u}C^{uu}_{{i^*},k}\Psi_{{i^*},k} [H^\top H]_{j,k}C^{uu}_{{i^*},j}\\
&\quad \text{(We make all terms negative except for $j=k={i^*}$)}\\
&\geq (C^{uu}_{{i^*},{i^*}})^2\left([H^\top H]_{{i^*},{i^*}}- \sum_{j\neq {i^*}}[H^\top H]_{j,{i^*}}-\sum_{k\neq {i^*}}\Psi_{{i^*},k} \sum_{j}[H^\top H]_{j,k}\right)\\
&\geq (C^{uu}_{{i^*},{i^*}})^2[H^\top H]_{{i^*},{i^*}}(1- (h_0+\psi_0(1+h_0))).
\end{align*}

\end{proof}

With Assumption \ref{aspt:obs} in place
and with suitable localization schemes that satisfy the assumption,
we can now prove an upper bound for the collapse
of the \emph{localized} EKI ensemble.
\begin{thm}
\label{thm:max}
Under Assumption \ref{aspt:obs}, for any $\delta>0$, there is a $t_0$ such that when $t\geq t_0$
\[
\|C^{uu}(t)\|_{\max}\leq \frac{M_C}{1+t}\text{, where } M_C:= \frac{\sqrt{1+8c_o \sigma}+1}{4c_o(1-\delta) }. 
\]
\end{thm}
\begin{proof}
We look at the diagonal entry of the covariance ODE  \eqref{eqn:CuuODE}
\[
\frac{d}{dt}C^{uu}_{i,i}(t)=-2\sum_{j=1}^{d_y}\Cloc^{up}_{i,j}(t) C^{up}_{i,j}(t)+\lambda_t.
\]
Let $i^*(t)$ be the index such that $C^{uu}_{i^*,i^*}(t)=\|C^{uu}(t)\|_{\max}$. Then Assumption \ref{aspt:obs} indicates that 
\[
\frac{d}{dt}C^{uu}_{i,i}(t)\leq -2c_oC^{uu}_{i,i}(t)+\lambda_t,\quad i=i^*(t). 
\]
Using a comparison principle Lemma \ref{lem:boundonderivative}, we can conclude that 
$\|C^{uu}_t\|_{\max}\leq y_t,$ where $y_t$ is the solution to a Riccati equation
\[
\dot{y}_t=-2c_o y_t^2+\frac{\sigma}{(t+1)^2},\quad y_0=\|C^{uu}(0)\|_{\max}.
\]
The solutions to  Riccati equations and their properties can be found in Lemma \ref{lem:ric}. In particular, we have our claim by finding 
\[
c_-=\frac{-1-\sqrt{1+8c_o\sigma}}{2},\quad \frac{c_-}{-a}=\frac{1+\sqrt{1+8c_o\sigma}}{4c_o}.
\]
\end{proof}

We note that Theorem \ref{thm:max} holds without inflation (we can set $\sigma=0$).
The theorem can indeed also hold without localization, 
provided Assumption \ref{aspt:obs} is satisfied by the unlocalized EKI.
The importance of the theorem,
however, is that the upper bound holds when localization is applied,
provided the localization is chosen to satisfy Assumption \ref{aspt:obs}.
When this is indeed the case, 
the LEKI satisfies this upper bound while exploring an enriched subspace,
larger than the subspace spanned by the initial ensemble
and, for that reason, can converge to a different solution than the unlocalized EKI.
Finally, we note that it is natural to set $\lambda_t=\sigma/(1+t)^2$, 
which is the same order as $\tfrac {d}{dt}(1/(1+t))$.

\subsection{Regularity and covariance lower bound}
We proceed to establish lower bounds for the ensemble covariance of LEKI
(as a second step towards a proof of the collapse of the LEKI ensemble).
In classical Kalman filter theory, posterior covariance lower bounds can be established when the system is controllable \cite{kalman1960new}.
For LEKI, controllability can be obtained by using additive inflation $\xi^j$ together with appropriate localization.
In particular, the additive inflation will lead to a lower bound for the diagonal terms of $C^{uu}$. 
Then using an appropriate  localization function $\psi$ can ensure that $\Cloc^{uu}$ is full rank. 
For our proof, we require 
the following sufficient regularity conditions.
\begin{aspt}[Regularity]
\label{aspt:regu}
There is an $L_R>0$ so that the following holds for all $i$
\[
L_R C^{uu}_{i,i}\|C^{uu}\|_{\max} \geq \sum_{j=1}^{d_y} \Cloc^{up}_{i,j}C^{up}_{i,j}. 
\]
In addition, the localization matrix is positive definite with $\psi_0=\lambda_{\min}(\Psi)>0$.  
\end{aspt}

Assumption \ref{aspt:regu} contains an assumption on the localization function,
which can be satisfied by choosing a sufficiently small localization radius $R_l$ in \eqref{eqn:uuloc}. 
Next, we show that  the first part of Assumption \ref{aspt:regu},
which connects $C^{uu}$ to the cross covariance $C^{up}$,
holds if  the observations are Lipschitz. 
For the centralized localization scheme, this results in the following Lemma.
\begin{lem}
\label{lem:RegCentral}
Suppose that each observation $G_j$ is component-wise Lipschitz: 
\[
|G_j(u)-G_j(v)|\leq \sum_k L_{j,k}|u_k-v_k| 
\]
while $\sum_k L_{j,k}\leq L$ for a constant $L$ and  all $j$, then Assumption \ref{aspt:regu} holds for the centralized localization scheme with 
\[
L_R=L^2\max_i\left(\sum_{j=1}^{d_y} \Psi_{i,i(j)}\right).
\]
\end{lem}
\begin{proof}
For centralized localization scheme, $C^{up}_{i,j}\Cloc^{up}_{i,j}=\Psi_{i,i(j)}(C^{up}_{i,j})^2$, where
\begin{align*}
(C^{up}_{i,j})^2&=\frac{1}{J^2(J-1)^2}\left(\sum_{m,n} (u^{m}_i-u^{n}_i) (G_j(u^{m})-G_j(u^{n}))\right)^2\\
&\leq \frac{1}{J^2(J-1)^2}\left(\sum_{m,n} |u^{m}_i-u^{n}_i| \left(\sum_k L_{j,k}|u_k^{m}-u_k^{n}|\right)\right)^2\\
&= \frac{1}{J^2(J-1)^2}\left( \sum_k \sum_{m,n} L_{j,k}|u^{m}_i-u^{n}_i| |u_k^{m}-u_k^{n}|)\right)^2\\
&\leq \frac{1}{J^2(J-1)^2}\left( \sum_k \sum_{m,n} L_{j,k}|u^{m}_i-u^{n}_i|^2\right)\left(\sum_k L_{j,k}\sum_{m,n} |u_k^{m}-u_k^{n}|^2\right)\\
&=\left( \sum_k  L_{j,k}C^{uu}_{i,i}\right)\left( \sum_k  L_{j,k}C^{uu}_{k,k}\right)\\
&\leq \left( \sum_k  L_{j,k}C^{uu}_{i,i}\right)\left( \sum_k  L_{j,k}\|C^{uu}\|_{\max}\right)
\leq L^2 C^{uu}_{i,i}\|C^{uu}\|_{\max}.
\end{align*}
\end{proof}

For the linearized localization scheme, we have the following lemma.
\begin{lem}
\label{lem:RegLinear}
Suppose the linearized localization scheme is applied with an $H$ such that
\[
\sum_{j=1}^{d_y} |H_{j,k}|\leq L,\quad\sum_{k=1}^{d_u} |H_{j,k}|\leq L,\quad \forall  k\leq d_u, j\leq d_y,
\]
then Assumption \ref{aspt:regu} holds with 
\[
L_R=L^2\max_i\left(\sum_{k=1}^{d_u} \Psi_{i,k}\right).
\]
\end{lem}
\begin{proof}
Note that 
\[
C^{up}_{i,j}=\sum_{k=1}^{d_u}C^{uu}_{i,k}H_{j,k},\quad \Cloc^{up}_{i,j}=\sum_{k=1}^{d_u}C^{uu}_{i,k}\Psi_{i,k} H_{j,k}.
\]
This leads to 
\begin{align*}
\sum_{j=1}^{d_y}\Cloc^{up}_{i,j} C^{up}_{i,j}&=\left(\sum_{k=1}^{d_u}C^{uu}_{i,k}\Psi_{i,k} H_{j,k}\right) \left(\sum_{k=1}^{d_u}C^{uu}_{i,k}H_{j,k}\right)\\
&\leq \sum_{j=1}^{d_y}\left(\sum_{k=1}^{d_u}\sqrt{C^{uu}_{i,i}C^{uu}_{k,k}}\Psi_{i,k} H_{j,k}\right) \left(\sum_{k=1}^{d_u}\sqrt{C^{uu}_{i,i}C^{uu}_{k,k}}H_{j,k}\right)\\
&\leq \sum_{j=1}^{d_y}\left(\sum_{k=1}^{d_u}\sqrt{C^{uu}_{i,i}\|C^{uu}\|_{\max}}\Psi_{i,k} H_{j,k}\right) \left(\sum_{k=1}^{d_u}\sqrt{C^{uu}_{i,i}\|C^{uu}\|_{\max}}H_{j,k}\right)\\
&\leq LC^{uu}_{i,i}\|C^{uu}\|_{\max}\sum_{k=1}^{d_u}\Psi_{i,k} \sum_{j=1}^{d_y} H_{j,k}\leq  L^2C^{uu}_{i,i}\|C^{uu}\|_{\max}\left(\sum_{k=1}^{d_u} \Psi_{i,k}\right).
\end{align*}
\end{proof}

Similar to our proof of the upper bound of the LEKI covariance matrix,
Lemmas \ref{lem:RegCentral} and \ref{lem:RegLinear} outline assumptions on the observation 
and localization functions that must be satisfied for the collapse of the localized EKI ensemble.
Specifically, with Assumption \ref{aspt:regu} in place and connected to the localization schemes,
we have the following theorem for a lower bound on the ensemble covariance of the LEKI ensemble. 
\begin{thm}
\label{thm:min}
Under Assumptions \ref{aspt:obs} and \ref{aspt:regu}, for any $\delta>0$, there is a $t_1$, so that when $t>t_1$
\[
C^{uu}_{i,i}(t)\geq \frac{m_c}{t+1} \quad\text{ for all index } i, 
\]
where
\[
m_c=\frac{\sigma(1-\delta)}{2L_R M_C-1}. 
\]
Moreover, 
\[
\lambda_{min} (\Cloc^{uu})\geq \psi_0\min_{i}C^{uu}_{i,i}\geq \frac{m_c\psi_0}{t+1}.
\]
\end{thm}
\begin{proof}
By Assumption \ref{aspt:regu}
\[
\frac{d}{dt}C^{uu}_{i,i}=-2\sum_{j=1}^{d_y}\Cloc^{up}_{i,j}C^{up}_{i,j} +\lambda_t\sigma^2\geq
-2L_R C^{uu}_{i,i}\|C^{uu}\|_{\max}+\lambda_t.
\]
By comparison principle, Lemma \ref{lem:boundonderivative},  we have 
$C^{uu}_{i,i}(t)\geq z_t,$ where $z_t$ is the solution to 
\[
\dot{z}_t=-2L_R \|C^{uu}(t)\|_{\max} z_t+\frac{\sigma}{(t+1)^2},\quad z_0=\|C^{uu}_0\|_{i,i}. 
\]
Using Duhamel's formula, we can write
\begin{align*}
z_t&=\exp\left(-2L_R\int^t_{t_0} \|C^{uu}(r)\|_{\max} dr\right) z_{t_0}+\int^t_{t_0} \exp\left(-2L_R\int^t_s \|C^{uu}(r)\|_{\max} dr\right) \frac{\sigma}{(s+1)^2}ds\\
&\geq \int^t_{t_0} \exp\left(-2L_R\int^t_s \|C^{uu}(r)\|_{\max} dr\right) \frac{\sigma}{(s+1)^2}ds\\
&\geq \int^t_{t_0} \exp\left(-2L_R \int^t_s \frac{M_C}{r+1}dr\right) \frac{\sigma}{(s+1)^2}ds,\quad \text{by Theorem \ref{thm:max}}\\
&=\sigma\int^t_{t_0}\frac{(s+1)^{2L_R M_C-2}}{(t+1)^{2L_R M_C}}ds=\frac{\sigma\left((t+1)^{2L_R M_C-1}-(t_0+1)^{2L_R M_C-1}\right)}{(2L_R M_C-1)(t+1)^{2L_R M_C}}.
\end{align*}
It can be seen that a threshold  $t_1$ exists. 

Finally, if $\Psi\succeq \psi_0 I_{d_u}$, we note that by Schur product theorem,
\[
\Cloc^{uu}=C^{uu}\circ \Psi\succeq C^{uu}\circ \psi_0 I_{d_u}=\psi_0 D^{uu}\succeq \psi_0\min_{i}C^{uu}_{i,i} I_{d_u}.
\] 
Here $D^{uu}$ is the diagonal part of $C^{uu}$. 
%
\end{proof}

In summary, Theorems \ref{thm:max} and \ref{thm:min}
show that the localized LEKI ensemble collapses at a controlled rate ($O(1/t)$).
The collapse occurs under Assumptions \ref{aspt:obs} and \ref{aspt:regu},
which are specified to localization schemes (centralized and linear\slash linearized)
in Lemmas \ref{lem:obsCentral} and \ref{lem:RegCentral}
about the upper\slash lower bound for centralized localization,
and in Lemmas \ref{lem:obsLinear} and \ref{lem:RegLinear}
the upper\slash lower bound for linear localization.

\section{Optimization guarantee}
\label{sec:opt}
In this section, we discuss whether LEKI can converge to the global minimizer of the loss function $l$. 
As is commonly done in optimization literature, we assume that the objective function $l$ is $c$-strongly convex: 
\begin{equation}
\label{eq:strong}
\|\nabla l(u)\|^2\geq  c (l(u)-l(u^*)),
\end{equation}
which guarantees the existence of a unique minimizer $u^*$.

Recall that the Kalman filter update can be viewed as a Gauss-Newton update \cite{bell1993iterated} and the EKI inherits this property \cite{chada2019convergence}. 
Specifically, the cross covariance matrix $C^{up}$ approximates $C^{uu} \nabla G(\bar{u}(t))^\top $.
We expect that similar ideas apply to LEKI, so that the performance of LEKI depends on how accurately  $\Cloc^{up}$ approximates $\Cloc^{uu} \nabla G(\bar{u}(t))^\top $. 
We then describe how this error decreases by studying the following error matrix:
\begin{equation}
\label{eq:Rt}
R(t)=\nabla G(\bar{u}(t)) \Cloc^{up}(t)-\nabla G(\bar{u}(t)) \Cloc^{uu}(t)\nabla G(\bar{u}(t))^\top.
\end{equation}
Roughly speaking, Theorem \ref{thm:globalloss} below shows that if  $\|R(t)\|$ decays faster than $1/t$, then the global loss of LEKI $l(\ubar(t))$ can decay like $1/{t}$ (if the convexity is strong). 
For uniform convergence, we focus on $l_j(u)=|G_j(u)-y_j|^2$, and Theorem \ref{thm:locopt} below shows that all $l_i(u)$ decay like $1/{t}$ uniformly, assuming that $\|R(t)\|_1$ decays faster than $1/t$. 

\subsection{Optimization performance}
We investigate the convergence of $\bar{u}_t$ to $u^*$ in terms of the corresponding loss function values under the following assumption: 
\begin{aspt}
\label{aspt:uloc}
Suppose  the following holds for some $\alpha\geq 0, R_\alpha\geq 0, L>0$
\begin{equation}
\label{aspt:error}
\|R(t)\|\leq R_\alpha \|C^{uu}(t)\|^{1+\alpha}_{\max},
\end{equation}
\begin{equation}
\label{aspt:regu1}
\|\nabla G (\ubar(t))^\top  \Cloc^{uu}  \nabla G(\ubar(t))\|\leq L \|C^{uu}(t)\|_{\max},
\end{equation}
\begin{equation}
\label{aspt:regu2}
\|G(\ubar(t))-\Gbar(t)\|^2\leq d_yL^2 \|C^{uu}(t)\|_{\max}. 
\end{equation}
\end{aspt}
Specifically, condition \eqref{aspt:error} requires the error matrix $R(t)$ to be smaller than the LEKI ensemble covariance. Given Theorems \ref{thm:max} and \ref{thm:min}, this mean that $\|R(t)\|$ must decays faster than $1/t$. Conditions \eqref{aspt:regu1} and \eqref{aspt:regu2} are regularity assumptions on the observation map $G$. 
Assumption \ref{aspt:uloc} is directly connected to the localization scheme we use within LEKI
and we will explain that further below, because the connections between
Assumption \ref{aspt:uloc} and localization re-appear (slightly modified) when discussing uniform convergence in Section~\ref{sec:UniformConvergence}.

\begin{thm}
\label{thm:globalloss}
Under Assumptions \ref{aspt:obs}, \ref{aspt:regu} and \ref{aspt:uloc}, suppose the loss function $f$ is $c$-strongly convex so \eqref{eq:strong} holds. Using $M_C$ and $m_c$ from Theorems \ref{thm:max} and \ref{thm:min}, we let
\[
c_\psi=\frac12m_c c\psi_0-2R_\alpha M_C1_{\alpha=0}. 
\]
Then for any $\epsilon>0$, 
\[
l(\ubar(t))-l(u^*)\lesssim \frac{l(\ubar_0)}{(t+1)^{c_\psi}}+ \frac{d_y}{(t+1)^{\min\{c_\psi,1-\epsilon\}}}\log(t+1)^{1_{c_\psi=1-\epsilon}}.
\]
\end{thm}
Note that $\alpha$ controls how well can the Jacobian matrix be approximated. 
If $\alpha>0$, $l(\ubar(t))-l(u^*)\to 0$. If $\alpha=0$, one needs $R_\alpha$ to be sufficiently small for asymptotic convergence.
And if the convexity is strong enough, we have approximately $O(\frac{d_y}{t})$ rate of convergence.  
\begin{proof}
For notational simplicity, we replace $l(u)$ with $l(u)-l(u^*)$ so that $l(u^*)=0$. 
We investigate the dynamics of $l(\ubar(t))$. 
\begin{align*}
\frac{d}{dt} &l(\ubar)=\langle\nabla l(\ubar(t)), \frac{d}{dt} \ubar(t)\rangle\\
&=-\langle\nabla l(\ubar), \Cloc^{up} (\Gbar(t)-y)\rangle\\
&=-2\langle (G(\ubar)-y), \nabla G(\ubar)^\top  \Cloc^{up} (\Gbar(t)-y)\rangle\\
&=-2\langle(G(\ubar)-y), \nabla G (\ubar)^\top  \Cloc^{uu}  \nabla G(\ubar) (\Gbar(t)-y)\rangle+
2\langle(G(\ubar)-y), R(t) (\Gbar(t)-y)\rangle\\
&=-2\langle(G(\ubar)-y), \nabla G (\ubar)^\top  \Cloc^{uu}  \nabla G(\ubar) (G(\ubar)-y)\rangle+
2\langle(G(\ubar)-y), R(t) (G(\ubar)-y)\rangle\\
&\quad-2\langle(G(\ubar)-y), \nabla G (\ubar)^\top  \Cloc^{uu}  \nabla G(\ubar) (\Gbar(t)-G(\ubar))\rangle
+2\langle(G(\ubar)-y), R(t) (G(\ubar)-\Gbar(t))\rangle.
\end{align*}
To continue, we bound each term above. The first one can be bounded by 
\begin{align*}
&-\langle(G(\ubar)-y), \nabla G (\ubar)^\top  \Cloc^{uu}  \nabla G(\ubar) (G(\ubar)-y)\rangle\\
&\quad =-\frac14\langle \nabla l(\ubar), \Cloc^{uu}  \nabla l(\ubar)\rangle\leq -\frac14\lambda_{\min}(\Cloc^{uu}) \|\nabla l(\ubar)\|^2\leq -\frac14c\lambda_{\min}(\Cloc^{uu})l(\ubar).
\end{align*}
The second term can be bounded using \eqref{aspt:error}
\[
|\langle(G(\ubar)-y) , R(t) (G(\ubar)-y)\rangle|\leq \| R(t)\| l(\bar{u})\leq R_\alpha \|C^{uu}\|^{1+\alpha}_{\max}l(\bar{u}).
\]
The third term  can be bounded using \eqref{aspt:regu1}
\begin{align*}
&|\langle(G(\ubar)-y), \nabla G (\ubar)^\top  \Cloc^{uu}  \nabla G(\ubar) (G(\ubar)-\Gbar(t))\rangle|\\
&\leq L\|G(\ubar)-y\|\|G(\ubar)-\Gbar(t)\|\|C^{uu}\|_{\max}\leq \frac12l(\ubar)\|C^{uu}\|^{1+\epsilon}_{\max}+\frac1{2}d_yL^2 \|C^{uu}\|_{\max}^{2-\epsilon}.
\end{align*}
The last term  can be bounded using \eqref{aspt:regu2}
\begin{align*}
|\langle(G(\ubar)-y), R(t) (G(\ubar)-\Gbar(t))\rangle|&\leq \|R(t)\|\|G(\ubar)-y\|\|G(\ubar)-\Gbar(t)\|\\
&\leq \frac12  l(\ubar)\|C^{uu}\|^{1+\epsilon}_{\max}+\frac1{2} d_yM^2_G R^2_\alpha\|C^{uu}\|^{2+2\alpha-\epsilon}_{\max}
\end{align*}
In summary we have
\[
\frac{d}{dt} l(\ubar)\leq -a_tl(\ubar)+b_t,
\]
where
\[
a_t=\frac12c\lambda_{\min}(\Cloc^{uu})-2R_\alpha \|C^{uu}\|^{1+\alpha}_{\max}-2\|C^{uu}\|^{1+\epsilon}_{\max},
\]
\[
b_t=d_y(L^2\|C^{uu}\|^{2-\epsilon}_{\max}+M_G^2\|C^{uu}\|^{2+2\alpha-\epsilon}_{\max}).
\]
By Theorems \ref{thm:max} and \ref{thm:min}, there is a certain $t_2>0$, when $t\geq t_2$,
$a_t\geq \frac{c_\psi}{1+t}$ while $b_t\lesssim \frac{d_y}{(1+t)^{2-\epsilon}}$.
Then we apply a Gronwall's inequality on $l(\ubar)$, Lemma \ref{lem:Gronwall} to obtain the final result. 
\end{proof}

\subsection{Uniform convergence with localized observations}
\label{sec:UniformConvergence}
In many inverse problems involve spatial models, we are interested in knowing whether the recovery is good uniformly over the space. 
To discuss this issue, we consider a simplified setting that has as many observations as parameters ($d_y=d_u$), and each $G_i(u)$ depends only on variables near $u_i$ (which we write as $u_I$).  In this context, it is natural to consider the data misfit of $G_i$, i.e. 
\[
l_i(u)= |G_i(u_I)-y_i|^2
\]
We now study if LEKI can minimize all $l_i$ uniformly over all $i$. 

As before, we assume that there is  a unique minimizer, $u^*$, of the loss function $l(u)=\sum_{i=1}^{d_u} l_i(u)$,
and, for this $u^*$, we have $\nabla l_i(u^*)=0$ for all $i$.
Since each $l_i$ depends only on variables ``near'' $u_i$, 
we re-define $c$-strong convexity locally.
\begin{aspt}
\label{aspt:lconvex}
There is a constant $c>0$, so that each $l_i(u)$ is locally $c$-strongly convex:
\[
\|\nabla l_i(u)\|^2 \geq c (l_i(u)-l_i(u^*)),\quad i=1,\ldots, d_u. 
\]
\end{aspt}
Note that if $l_i$ is $c$-strongly convex in the standard sense, then it is also locally $c$-strongly convex. But locally strongly convex can be more general. In particular, if $l_i$ only depends on some part of components, say $l_i(u)=l_i(u_I)$, then for Assumption \ref{aspt:lconvex} to hold, $l_i(u)$ only needs to be $c$-strongly convex in $u_I$ but not in all components.

For our proof of uniform convergence, we modify Assumption \ref{aspt:uloc}.
Recall that Assumption \ref{aspt:uloc} is used to show convergence in $l_2$ norms,
so that the conditions involve in $l_2$ and $l_2$-operator norms. For uniform convergence,
we formulated the assumption in $l_\infty$-related norms and the associated operator norms. 
\begin{aspt}
\label{aspt:uloc2}
Suppose  the following holds for some $\alpha\geq 0, R_\alpha\geq 0, L>0$
\begin{equation}
\label{aspt:regu12}
\nabla G_j (\ubar)^\top  \Cloc^{uu}  \nabla G_j(\ubar)\leq L \|C^{uu}(t)\|_{\max},\quad \forall j=1,\ldots, d_y,
\end{equation}
\begin{equation}
\label{aspt:regu22}
|G_j(\ubar)-\Gbar_j(t)|^2\leq L^2 \|C^{uu}(t)\|_{\max},\quad \forall j=1,\ldots, d_y. 
\end{equation}
\end{aspt}

The final assumption we need is that $l_i$  decorrelates with $l_j$ when $i$ and $j$ are far apart
(this is one of the essential assumptions for localization to work). 
The correlation between two functions $l_i(u)$ and $l_j(u)$ can be measured by 
$\nabla l_i(u)^\top  C_u\nabla l_j(u)$, where $C_u$ is the covariance matrix of random vector $u$. In the context of LEKI, it is natural to replace $C_u$ with estimator $\Cloc^{uu}$ and make the following assumption: 
\begin{aspt}
\label{aspt:ploc}
There is a symmetric matrix $\phi\in \reals^{d_u\times d_u}$ and $R_\alpha\geq 0$
such that the following hold
\begin{equation}
\label{aspt:error2}
[R(t)]_{i,j}\leq R_\alpha \phi_{i,j}\|C^{uu}(t)\|^{1+\alpha}_{\max},
\end{equation}
\[
\frac12\phi_{i,j}\left(\|\nabla G_j(u)\|^2_{\Cloc^{uu}}+\|\nabla G_i(u)\|^2_{\Cloc^{uu}}\right)
\geq  \nabla G_j^\top  \Cloc^{uu}\nabla G_i.
\]
Moreover, we assume there is a $\phi_0>0$ so that  $\sum_{j\neq i}\phi_{i,j}<1-\phi_0.$ 
\end{aspt}
We remark that this condition can be difficult to check in general,
but if each observation concerns only one component, i.e. $l_j(u)=l_j(u_j)$, Assumption \ref{aspt:ploc} holds with $\phi_{i,j}=\Psi_{i,j}$, because we can apply Cauchy Schwarz using
\[
\|\nabla l_j(u)\|^2_{\Cloc^{uu}}=|\dot {l}_j(u_j)|^2C^{uu}_{j,j},\quad  \nabla l_j^\top  \Cloc^{uu}\nabla l_i=\dot {l}_j(u_j)\dot {l}_i(u_i)C^{uu}_{i,j}\Psi_{i,j}. 
\]

We can now state and prove the theorem about uniform convergence of LEKI.

\begin{thm}
\label{thm:locopt}
Suppose all $l_j(u)$ are $c$-strongly convex and the localization matrix $\Psi$ satisfies $\lambda_{\min}(\Psi)\geq \psi_0$,  under Assumptions \ref{aspt:obs}, \ref{aspt:regu},  \ref{aspt:lconvex}, \ref{aspt:uloc2}, and \ref{aspt:ploc}, the following holds:
\[
\max_j |l_j(\ubar(t))-l_j(u^*)|\lesssim \frac{\max_j|l_j(\ubar(0))-l_j(u^*)|}{(t+1)^{c_\psi}}+\frac{1}{(t+1)^{\min\{c_\psi,1-\epsilon\}}}
\]
where 
\[
c_\psi=\frac14c\psi_0\phi_0 m_c-9R_\alpha  M_C1_{\alpha=0}. 
\]
\end{thm}

\begin{proof}
We investigate the dynamics of  $l_j(u)=|G_j(u)-y_j|^2$, which follows
\begin{align*}
\frac{d}{dt} l_j(\ubar)&=\langle\nabla l_j(\ubar), \frac{d}{dt} \ubar\rangle\\
&=-\langle\nabla l_j(\ubar), \Cloc^{up} (\Gbar(t)-y)\rangle\\
&=-2 (G_j(\ubar)-y_j) \nabla G_j(\ubar)^\top \Cloc^{up} (\Gbar(t)-y) =2\mathcal{A}+2\mathcal{B}.
\end{align*}
Here we let
\[
 \mathcal{A}=-(G_j(\ubar)-y_j) R_{j,\cdot}(\Gbar(t)-y),
\]
where $R(t)$ is the matrix defined in \eqref{eq:Rt} and $R_{j,\cdot}$ is its $j$th row, 
and
\begin{align*}
\mathcal{B}&=-(G_j(\ubar)-y_j) \nabla G_j(\ubar)^\top \Cloc^{uu}\nabla G (\ubar)(\Gbar(t)-y)\\
&=-\frac14\langle\nabla l_j(\ubar), \Cloc^{uu} \nabla l_j(\ubar)\rangle-\sum_{i\neq j} (G_j(\ubar)-y_j)(G_i(\ubar)-y_i)\nabla G_j(\ubar)^\top  \Cloc^{uu} \nabla G_i(\ubar)\\
&\quad-\sum_{i\neq j}(G_j(\ubar)-y_j)(\Gbar_i(t)-G_i(\ubar))\nabla G_j(\ubar)^\top \Cloc^{uu}\nabla G_i (\ubar).
\end{align*}
We use the following to bound $\mathcal{A}$
\begin{align*}
&|(G_j(\ubar)-y_j) R_{j,\cdot}(\Gbar(t)-y)|\\
&\leq|\sum_{i=1}^{d_u} R_{j,i}(G_j(\ubar)-y_j)(G_i(\ubar)-y_i)|+|\sum_{i=1}^{d_u} R_{j,i}(G_j(\ubar)-y_j)(\Gbar_i-G_i(\ubar))|\\
&\leq \sum_{i=1}^{d_u} |R_{j,i}|\sqrt{l_il_j}+\sum_{i=1}^{d_u} R_{j,i}  \sqrt{l_j}  L\|C^{uu}\|_{\max}\\
&\leq \sum_{i=1}^{d_u}|R_{j,i}|(l_i+l_j)+\sum_{i=1}^{d_u} |R_{j,i}| l_j+L^2\|C^{uu}\|_{\max}^2\sum_{i=1}^{d_u} |R_{j,i}|\\
&\leq R_\alpha\|C^{uu}\|^{1+\alpha}_{\max} \left(3l_j+\sum_{i\neq j}\phi_{j,i}l_i\right)+L^2R_\alpha \|C^{uu}\|_{\max}^{3+\alpha}.
\end{align*}
To bound $\mathcal{B}$, we apply Assumption \ref{aspt:ploc} to the second term in the decomposition of $\mathcal{B}$ and find
\begin{align*}
&- (G_j(\ubar)-y_j)(G_i(\ubar)-y_i)\nabla G_j(\ubar)^\top  \Cloc^{uu} \nabla G_i(\ubar)\\
&\leq \sqrt{l_j l_i} |\nabla G_j(\ubar)^\top  \Cloc^{uu} \nabla G_i(\ubar)|\\
&\leq  \phi_{i,j}\sqrt{l_j l_i} \|\nabla G_j\|_{\Cloc^{uu}}\|\nabla G_i\|_{\Cloc^{uu}}\\
& \leq  \frac18\phi_{i,j}(\|\nabla l_j(\ubar)\|^2_{\Cloc^{uu}}+\|\nabla l_i(\ubar)\|^2_{\Cloc^{uu}}).
\end{align*}
To bound the third term in the decomposition of $\mathcal{B}$, we first note that
\begin{align*}
&|\langle(G_j(\ubar)-y_j),\nabla G_j(\ubar)^\top \Cloc^{uu}\nabla G_i (\ubar)(\Gbar_i(t)-G_i(\ubar))\rangle|\\
&\leq \sqrt{l_j}|\Gbar_i(t)-G_i(\ubar)||\nabla G_j^\top \Cloc^{uu}\nabla G_i |\leq \phi_{i,j} \sqrt{l_j}|\Gbar_i(t)-G_i(\ubar)|\|\nabla G_j\|_{\Cloc^{uu}}\|\nabla G_i\|_{\Cloc^{uu}}\\
&\leq \frac12\phi_{i,j} \sqrt{L}\|\nabla l_j\|_{\Cloc^{uu}}\| \|\Cloc^{uu}\|^{1/2}_{\max}|\Gbar_i(t)-G_i(\ubar)|,
\end{align*}
where in the last line we used the fact that $\nabla l_j=2\sqrt{l_j}\nabla G_j$ and
\[
\|\nabla G_i\|^2_{\Cloc^{uu}}=[\nabla G^\top \Cloc^{uu}\nabla G]_{i,i}\leq L\|\Cloc^{uu}\|_{\max}
\]
by \eqref{aspt:regu1}. This leads to 
\begin{align*}
&\sum_{i\neq j}|G_j(\ubar)-y_j)(\Gbar_i(t)-G_i(\ubar))\nabla G_j(\ubar)^\top \Cloc^{uu}\nabla G_i (\ubar)|\\
&\leq \frac12 \sqrt{L}\sum_{i\neq j}\phi_{i,j} \|\nabla l_j\|_{\Cloc^{uu}}\| \|\Cloc^{uu}\|^{1/2}_{\max}|\Gbar_i(t)-G_i(\ubar)|\\
&\leq \frac12 L\sum_{i\neq j}\phi_{i,j} \|\nabla l_j\|_{\Cloc^{uu}}\|  \|C^{uu}\|_{\max}\\
&\leq \frac12 L\|\nabla l_j\|_{\Cloc^{uu}}\|\|C^{uu}\|^{1}_{\max}\leq \|\nabla l_j\|^2_{\Cloc^{uu}} \|C^{uu}\|^{\epsilon}_{\max}+L^2\|C^{uu}\|^{2-\epsilon}_{\max}.
\end{align*}
%
Putting all these estimates back into the decomposition of $\mathcal{B}$,  we find
\begin{align*}
&-\langle(G_j(\ubar)-y_j) , \nabla G_j(\ubar)^\top \Cloc^{uu}\nabla G (\ubar)(\Gbar(t)-y)\rangle\\
&\leq 
(-\frac14+\|C^{uu}\|^\epsilon_{\max})\|\nabla l_j(\ubar)\|^2_{\Cloc^{uu}}+\frac18\sum_{i\neq j}\phi_{i,j}\left(\|\nabla l_j(\ubar)\|^2_{\Cloc^{uu}}+\|\nabla l_i(\ubar)\|^2_{\Cloc^{uu}}\right)+ L^2\|\Cloc^{uu}\|^{2-\epsilon}_{\max}\\
&\leq 
(-\frac18-\frac18\phi_0+\|C^{uu}\|^\epsilon_{\max})\|\nabla l_j(\ubar)\|^2_{\Cloc^{uu}}+\frac18\sum_{i\neq j}\phi_{i,j}\|\nabla l_i(\ubar)\|^2_{\Cloc^{uu}}+ L^2\|\Cloc^{uu}\|^{2-\epsilon}_{\max}.
\end{align*}
In summary,
\begin{align*}
\frac{d}{dt} l_j(\ubar)
&=2\mathcal{A}+2\mathcal{B}\\
&\leq (-\frac14-\frac14\phi_0+2\|C^{uu}\|^\epsilon_{\max})\|\nabla l_j(\ubar)\|^2_{\Cloc^{uu}}+\frac14\sum_{i\neq j}\phi_{i,j}\|\nabla l_i(\ubar)\|^2_{\Cloc^{uu}}+ 2L^2\|\Cloc^{uu}\|^{2-\epsilon}_{\max}\\
&\quad\quad +2R_\alpha\|C^{uu}\|^{1+\alpha}_{\max} \left(3l_j+\sum_{i\neq j}\phi_{j,i}l_i\right)+2L^2R_\alpha \|C^{uu}\|_{\max}^{3+\alpha}.
\end{align*}
A technical Lemma \ref{lem:vexist} indicates that there is a set of numbers $v^k_j$ such that 
\begin{enumerate}[1)]
\item $v^{k}_j\geq 0, \forall j$ and in specific $v^i_i\geq 1-\phi_0$.
\item For all index $j$, $\sum_{l\neq j} \phi_{j,l}v^i_l\leq v^i_j$.
\item $\sum_{j=1}^{N_x} v^i_j\leq 1$. 
 \end{enumerate}
Now let $L_k(u)=\sum_{j=1}^{d_u} v^k_{j} (l_j(u)-l_j(u^*))$. Then using the properties of $v^k_j$ above
\begin{align*}
\frac{d}{dt} L_k(\bar{u})&\leq \sum_{j=1}^{d_u} v^k_j\frac{d}{dt} l_j(\ubar)\\
&\leq \sum_{j=1}^{d_u}v^k_j\bigg((-\frac14-\frac14\phi_0+2\|C^{uu}\|^\epsilon_{\max})\|\nabla l_j(\ubar)\|^2_{\Cloc^{uu}}+\frac14\sum_{i\neq j}\phi_{i,j}\|\nabla l_i(\ubar)\|^2_{\Cloc^{uu}}\\
&\quad\quad\quad+ 2L^2\|\Cloc^{uu}\|^{2-\epsilon}_{\max}
 +2R_\alpha\|C^{uu}\|^{1+\alpha}_{\max} \bigg(3l_j+\sum_{i\neq j}\phi_{j,i}l_i\bigg)+2L^2R_\alpha \|C^{uu}\|_{\max}^{3+\alpha}\bigg)\\
&= \sum_{i=1}^{d_u}\bigg((-\frac14 v^k_i-\frac14\phi_0 v_i^k+\frac14\sum_{j\neq i} \phi_{i,j} v^k_j+2\|C^{uu}\|^\epsilon_{\max}v_i^k)\|\nabla l_i(\ubar)\|^2_{\Cloc^{uu}}\\
&\quad\quad\quad + 2L^2v^k_j\|\Cloc^{uu}\|^{2-\epsilon}_{\max}+R_\alpha\|C^{uu}\|^{1+\alpha}_{\max} \bigg(6v^k_i+3\sum_{j\neq i}\phi_{i,j}v^k_j\bigg)l_i+2L^2v^k_iR_\alpha \|C^{uu}\|_{\max}^{3+\alpha}\bigg)\\
&\leq \sum_{i=1}^{d_u} \bigg((- 2\phi_0+2\|C^{uu}\|^\epsilon_{\max})v^k_i \|\nabla l_i(\ubar)\|^2_{\Cloc^{uu}}+9R_\alpha \|C^{uu}\|^{1+\alpha}_{\max}l_i\bigg)\\
&\quad\quad\quad+
2L^2(\|C^{uu}\|_{\max}^{2-\epsilon}+R_\alpha \|C^{uu}\|_{\max}^{3+\alpha}).
\end{align*}
We note that by strong convexity,
\[
\|\nabla l_i(\ubar)\|^2_{\Cloc^{uu}}\geq l_i(\ubar) \Rightarrow \sum_i v_i^k\|\nabla l_i(\ubar)\|^2_{\Cloc^{uu}}\geq 
\sum_i v_i^k c l_i(\ubar)\lambda_{\min}(\Cloc^{uu}). 
\]
So we can further deduce that 
\begin{align*}
\frac{d}{dt} L_k(\bar{u})&\leq \sum_{i=1}^{d_u} \bigg(- 2c\phi_0\lambda_{\min}(\Cloc^{uu})+2c\|C^{uu}\|^\epsilon_{\max}\lambda_{\min}(\Cloc^{uu})+9R_\alpha \|C^{uu}\|_{\max}^{1+\alpha}\bigg)v^k_i l_i\\
&\quad\quad+
2L^2(\|C^{uu}\|_{\max}^{2-\epsilon}+R_\alpha \|C^{uu}\|_{\max}^{3+\alpha})\\
&=  \bigg( - \frac14c\phi_0\lambda_{\min}(\Cloc^{uu})+2c\|C^{uu}\|^\epsilon_{\max}\lambda_{\min}(\Cloc^{uu})+9R_\alpha \|C^{uu}\|_{\max}^{1+\alpha}\bigg)L_k\\
&\quad\quad+
2L^2(\|C^{uu}\|_{\max}^{2-\epsilon}+R_\alpha \|C^{uu}\|_{\max}^{3+\alpha}).
\end{align*} 
In summary we have
\[
\frac{d}{dt} L_k(\ubar)\leq -a_tL_k(\ubar)+b_t,
\]
where
\[
a_t=\frac14c\phi_0\lambda_{\min}(\Cloc^{uu})-2c\|C^{uu}\|^\epsilon_{\max}\lambda_{\min}(\Cloc^{uu})-9R_\alpha \|C^{uu}\|_{\max}^{1+\alpha},
\]
and 
\[
b_t=2L^2(\|C^{uu}\|_{\max}^{2-\epsilon}+R_\alpha \|C^{uu}\|_{\max}^{3+\alpha}). 
\]
By Theorems \ref{thm:max} and \ref{thm:min}, we find that there is a threshold time $t_2$ so that for all $t\geq t_2$
\[
a_t\geq \frac{c\psi_0\phi_0 m_c}{4(t+1)} -\frac{9R_\alpha  M_C}{(1+t)}1_{\alpha=0}=\frac{c_\psi}{t+1}.
\]
while
\[
b_t\leq \frac{2L^2 M_C^{2-\epsilon}}{(1+t)^{2-\epsilon}}. 
\]
So  we can apply a Gronwall's inequality, Lemma \ref{lem:Gronwall}, to obtain 
\[
L_k(t)\lesssim \frac{L_k(0)}{(t+1)^{c_\psi}}+\frac{1}{(t+1)^{\min \{c_\psi,1-\epsilon\}}}.
\]
Finally, we note that 
\[
L_k(t)\geq v_k^k l_k(t)\geq \phi_0l_k(t),\quad \max_{j}l_j(0)\geq L_k(0).
\]
So an upper bound of $L_k(t)$ leads to an upperbound of $l_k(t)$. 

\end{proof}

\subsection{Connecting Assumptions \ref{aspt:uloc} and \ref{aspt:uloc2} to localization}
\label{sec:Rcond}
We explain Assumptions \ref{aspt:uloc} and \ref{aspt:uloc2} and describe
how localization can be used to satisfy these assumptions and, for that reason,
obtain convergence guarantees for LEKI.
Each assumption contains three conditions,
\eqref{aspt:error},  \eqref{aspt:regu1} and \eqref{aspt:regu2} (for Assumption  \ref{aspt:uloc}),
and \eqref{aspt:error2}, \eqref{aspt:regu12} and \eqref{aspt:regu22} (for Assumption  \ref{aspt:uloc2})

We start with the regularity conditions \eqref{aspt:regu1} and \eqref{aspt:regu2}, or  \eqref{aspt:regu12} and \eqref{aspt:regu22} (uniform convergence).
These conditions do not use $\Cloc^{up}$, so that the following discussion holds for both the centralized and linearized localization schemes.
\begin{lem}
Suppose that there is a matrix $Q\in \reals^{d_u\times d_y}$ so that 
\[
|\partial_{u_i} G_j(u)|\leq Q_{i,j}.
\]
Then \eqref{aspt:regu1} and \eqref{aspt:regu12} hold with  $L\leq \|\Psi\|\|Q\|^2$.
\end{lem}
\begin{proof}
Note that for any $y\in R^{d_y}$, the following upper bound holds
\begin{align*}
y^\top \nabla G (\ubar)^\top  \Cloc^{uu}  \nabla G(\ubar) y&=\sum_{i,j=1}^{d_y}\sum_{m,n=1}^{d_u} y_iy_j\partial_{u_m}G_i(u)\partial_{u_n}G_j(u)C^{uu}_{m,n}\Psi_{m,n}\\
&\leq  \|C^{uu}\|_{\max}\sum_{i,j=1}^{d_y}\sum_{m,n=1}^{d_u} y_iy_jQ_{i,m}Q_{j,n}\Psi_{m,n}\\
&= \|C^{uu}\|_{\max} y^\top Q^\top  \Psi Qy\leq \|C^{uu}\|_{\max} \|\Psi\| \|Qy\|^2.
\end{align*}
Therefore, \eqref{aspt:regu1} hold by definition of $l_2$ operator norm. Choosing $y$ to be the $i$-th Euclidean basis vector leads to \eqref{aspt:regu12}. 
\end{proof}

\begin{lem}
Suppose there are matrices $Q^j$ that dominate the Hessian of $G_j$:
\[
|[H_{G_j}]_{m,n}|\leq Q^j_{m,n}.
\]
Suppose $Q^j$  are sparse so that there is a constant $L_G$ that satisfies:
\[
\sum_{m,n=1}^{d_u}|Q_{m,n}^j|\leq L_G.
\]
Then \eqref{aspt:regu2} and  \eqref{aspt:regu22} hold with $L=L_G$:
\[
\|G(\ubar)-\Gbar(t)\|^2=\sum_{j=1}^{d_y}|\Gbar_j(u)-G_j(\ubar)|^2\leq d_yL^2_G \|C^{uu}\|^2_{\max}. 
\]
\[
|\Gbar_j(u)-G_j(\ubar)|^2\leq L_G^2\|C^{uu}\|_{\max}\quad \text{for all }j.
\]
\end{lem}
\begin{proof}
Using Taylor expansion, we find that 
\[
|G_j(u^i)-G_j(\ubar)-\nabla G_j(\ubar) (u^i-\ubar)|\leq \frac12(u^i-\ubar)^\top Q^j (u^i-\ubar).
\]
Therefore we can obtain \eqref{aspt:regu22} through
\begin{align*}
|\Gbar_j(u)-G_j(\ubar)|&=\left|\frac{1}{J}\sum_{i=1}^J(G_j(u^i)-G_j(\ubar)-\nabla G_j(\ubar) (u^i-\ubar))\right|\\
&\leq\frac{1}{J}\sum_{i=1}^J \left| G_j(u^i)-G_j(\ubar)-\nabla G_j(\ubar) (u^i-\ubar))\right|\\
&\leq  \frac{1}{2J}\sum_{i=1}^J(u^i-\ubar)^\top Q^j (u^i-\ubar)\\
&=  \frac{1}{2J}\sum_{i=1}^J\text{tr}(Q^j (u^i-\ubar)(u^i-\ubar)^\top )\\
&=\frac{J-1}{2J}\text{tr}(Q^jC^{uu})\leq \frac{J-1}{2J}\|C^{uu}\|_{\max}\sum_{m,n}|Q^j_{m,n}|
\end{align*}
This also leads to \eqref{aspt:regu2} 
\[
\|G(\ubar)-\Gbar(t)\|^2=\sum_{j=1}^{d_y}|\Gbar_j(u)-G_j(\ubar)|^2\leq d_yL^2_G \|C^{uu}\|^2_{\max}.
\]
\end{proof}

Next we explain conditions \eqref{aspt:error} and \eqref{aspt:error2}  for the centralized localization scheme. 
\begin{lem}
For centralized localization scheme, suppose $G_j$ is depends mostly on one component $u_{i(j)}$:
\[
\sum_{m=1}^{d_u}|\partial_{u_m} G_j ||\Psi_{i,i(j)}-\Psi_{i,m}|\leq R_0\Psi_{i,i(j)},
\]
and
\[
\max_{i}\sum_{j=1}^{d_y}\sum_{k=1}^{d_u}|\partial_{u_k} G_{i}| \Psi_{k,i(j)}\leq L_3\quad \max_{j}\sum_{i=1}^{d_y}\sum_{k=1}^{d_u}|\partial_{u_k} G_{i}| \Psi_{k,i(j)}\leq L_3.
\]
Suppose also there is are matrices $Q^j$ that dominate the Hessian of $G_j$:
\[
|[H_{G_j}]_{m,n}|\leq Q^j_{m,n},
\]
and $Q^j$  are sparse so that there is a constant $L_G$ satisfies:
\[
\sum_{m,n=1}^{d_u}|Q_{m,n}^j|\leq L_G.
\]
Then \eqref{aspt:error} and \eqref{aspt:error2} holds with 
\begin{align*}
\left\|R \right\|&\leq \frac{J L_3}{J-1}R_0\|C^{uu}\|_{\max}+\sqrt{J-1}L_3L_G\|C^{uu}\|^{3/2}_{\max},\\
\left|[R]_{i,j}\right|&\leq \frac{J \Psi_{i,i(j)}}{J-1}R_0\|C^{uu}\|_{\max}+\sqrt{J-1}\Psi_{i,i(j)} L_G\|C^{uu}\|^{3/2}_{\max}.
 \end{align*}
 In particular, if each $G_j$ consists of only one location, i.e. $ G_j (u)=G_j(u_{i(j)})$,
 \[
|\partial_{u_m} G_j ||\Psi_{i,i(j)}-\Psi_{i,m}|\equiv 0\text{ and } R_0=0. 
 \]
 \end{lem}

\begin{proof}
Denote
\[
R^j_{k}=G_j(u^k)-G_j(\ubar)-\nabla G_j (\ubar)(u^k-\ubar)^\top 
\]
Then by Taylor expansion
\[
|R^j_{k}|\leq \frac12 (u^k-\ubar)^\top Q^j (u^k-\ubar). 
\]
Then 
\[
(G_j(u^k)-G_j(\ubar))(u_i^k-\ubar_i)= \nabla G_j (\ubar)(u^k-\ubar)^\top  (u^k_i-\ubar_i)+R^j_k(u^k_i-\ubar_i). 
\]
Note that 
\[
\sum_{k=1}^J(\Gbar_j-G_j(\ubar))(u^k-\ubar)^\top =(\Gbar_j-G_j(\ubar))\sum_{k=1}^J (u^k-\ubar)^\top =0. 
\]
Therefore 
\begin{align*}
[\Cloc^{up}]_{i,j}=\frac{1}{J-1}\sum_{k=1}^J\sum_{m=1}^{d_u}\partial_{u_m} G_j (\ubar)(u^k_m-\ubar_m) (u^k_i-\ubar_i)\Psi_{i,i(j)}+R^j_k(u^k_i-\ubar_i)\Psi_{i,i(j)}.
\end{align*}
Note also that
\begin{align*}
[\Cloc^{uu}\nabla G^\top ]_{i,j}=\frac{1}{J-1}\sum_{k=1}^J\sum_{m=1}^{d_u}\partial_{u_m} G_j (\ubar)(u^k_m-\ubar_m) (u^k_i-\ubar_i)\Psi_{i,m}.
\end{align*}
In summary,
\begin{align*}
&\left|[\Cloc^{up}]_{i,j}-[ \Cloc^{uu} \nabla G^\top ]_{i,j}\right|\\
&=\left|\frac{1}{J-1}\sum_{k=1}^J\sum_{m=1}^{d_u}\partial_{u_m} G_j (\ubar)(u^k_m-\ubar_m) (u^k_i-\ubar_i)(\Psi_{i,i(j)}-\Psi_{i,m})+R^j_k(u^k_i-\ubar_i)\Psi_{i,i(j)}\right|\\
&\leq \frac{J\Psi_{i,i(j)}}{J-1}\|C^{uu}\|_{\max}\sum_{m=1}^{d_u}|\partial_{u_m} G_j (\ubar)||1-\Psi_{i,m}/\Psi_{i,i(j)}|+\frac{\Psi_{i,i(j)}}{J-1}\sum_{k=1}^J\sum_{m=1}^{d_u} R^j_k|u^k_i-\ubar_i|\\
&\leq \frac{J \Psi_{i,i(j)}}{J-1}R_0\|C^{uu}\|_{\max}+\frac{\Psi_{i,i(j)}}{J-1}\sum_{k=1}^J\sum_{m=1}^{d_u} R^j_k|u^k_i-\ubar_i|.
\end{align*}
Note that 
\[
|u^k_i-\ubar_i|^2\leq \sum_{k=1}^J|u^k_i-\ubar_i|^2=(J-1)C^{uu}_{i,i}\leq (J-1)\|C^{uu}\|_{\max}.
\]
Meanwhile
\begin{align*}
\frac{1}{J-1}\sum_k R^j_{k}&\leq \frac{1}{2(J-1)}\sum_{i=1}^J(u^i-\ubar)^\top Q^j (u^i-\ubar)\\
&=  \frac{1}{2(J-1)}\sum_{i=1}^J\text{tr}(L^j (u^i-\ubar)(u^i-\ubar)^\top )\\
&=\frac{1}{2}\text{tr}(L^jC^{uu})\leq \frac{1}{2}\|C^{uu}\|_{\max}\sum_{m,n}|Q^j_{m,n}|.
\end{align*}
We have shown that 
\[
\left|[\Cloc^{up}]_{i,j}-[ \Cloc^{uu} \nabla G^\top ]_{i,j}\right|\leq \frac{J \Psi_{i,i(j)}}{J-1}R_0\|C^{uu}\|_{\max}+\sqrt{J-1}\Psi_{i,i(j)} L_G\|C^{uu}\|^{3/2}_{\max}=:\Psi_{i,i(j)}R_1
\]
Finally, we have that 
\begin{align*}
\left\|\nabla G \Cloc^{up}-\nabla G  \Cloc^{uu} \nabla G^\top \right\|_1
&\leq \max_{i}\sum_{j,k}|\partial_{u_k} G_{i}| \left|[\Cloc^{up}]_{k,j}-[ \Cloc^{uu} \nabla G^\top ]_{k,j}\right|\\
&\leq\max_{i}\sum_{j,k}|\partial_{u_k} G_{i}| R_1 \Psi_{k,i(j)}\leq L_3R_1. 
\end{align*}
Likewise
\begin{align*}
\left\|[\nabla G \Cloc^{up}]^\top -\nabla G  \Cloc^{uu} \nabla G^\top \right\|_1
&\leq \max_{j}\sum_{i,k}|\partial_{u_k} G_{i}| \left|[\Cloc^{up}]_{k,j}-[ \Cloc^{uu} \nabla G^\top ]_{k,j}\right|\\
&\leq\max_{j}\sum_{i,k}|\partial_{u_k} G_{i}| R_1 \Psi_{k,i(j)}\leq L_3R_1. 
\end{align*}
We use Lemma \ref{lem:norm} to obtain an upper bound for $\left\|\nabla G \Cloc^{up}-\nabla G  \Cloc^{uu} \nabla G^\top \right\|$. 
\end{proof}

Finally, we discuss condition \eqref{aspt:error} in the context of a linearized localization scheme.
Specifically, if we can compute the the Jacobian matrix $\nabla G(\ubar)$, we can use $H=\nabla G(\ubar)$ and \eqref{aspt:error} holds with $R_\alpha=0$. Without the Jacobian, one can consider a finite difference estimator of $\nabla G$ with step size $O(\frac{1}{(t+1)^\beta})$, so that $\|H-\nabla G\|=O(\frac{1}{(t+1)^\beta})$ and \eqref{aspt:error} holds with $\alpha=\beta+1$.

\begin{lem}
\label{lem:errorloc}
For linearized localization scheme, the following holds:
\begin{align*}
&\|[\nabla G(\bar{u}) \Cloc^{up}-[\nabla G(\bar{u}) \Cloc^{uu}\nabla G(\bar{u})^\top]\|\leq \|H-\nabla G(\ubar)\|\|\nabla G(\ubar)\|\|C^{uu}\|_{\max}\psi_0,
\end{align*}
where $\psi_0=\max_i\sum_j |\Psi_{i,j}|$. 
\end{lem}
\begin{proof}
Simply note that under linearized scheme,
\begin{align*}
&\quad\|[\nabla G(\bar{u}) \Cloc^{up}-[\nabla G(\bar{u}) \Cloc^{uu}\nabla G(\bar{u})^\top]\|\\
&=\|[\nabla G(\bar{u}) \Cloc^{uu}(H-\nabla G(\bar{u}))^\top]\|\\
&\leq \|H-\nabla G\|\|\nabla G\|\|\Cloc^{uu}\|.
\end{align*}
Then we use the fact that 
\[
\|\Cloc^{uu}\|\leq \max_i\sum_j |\Cloc^{uu}_{i,j}|\leq \|\Cloc^{uu}\|_{\max}\psi_0. 
\]

\end{proof}

\section{Numerical illustrations of localization in EKI}
\label{sec:num}
We illustrate the use of localization in EKI in numerical experiments,
which range from simple linear and nonlinear toy problems, to Lorenz models
and actual field-data.
For a given problem, we perform an EKI and a LEKI 
and compute the associated data misfit defined by
\begin{equation}
\label{eq:Misfit}
	\text{Misfit} = \left(\frac{1}{d_y}\sum_{i=1}^{d_y} (y_i-G_i(\bar{u}))^2\right)^{0.5}.
\end{equation}
We also consider the maximum error defined by
\begin{equation}
	\text{Max. Error} = \max_i \vert y_i-G_i(\bar{u})\vert,
\end{equation}
the trace of the ensemble covariance (tr$(C^{uu})$),
and the max.\slash min.~of its diagonal elements ($\|C^{uu}\|_{\max}$, $\|C^{uu}\|_{\min}$)
after each step in the iteration. 
In the numerical experiments below, 
we perform a set number of iterations with EKI and LEKI
because we are mostly interested in demonstrating the collapse,
not so much in when to stop an EKI or LEKI (which is also an interesting topic).
Because the initial ensemble 
has some effect on the output 
(error and ensemble covariances),
we randomize all numerical experiments
and perform all calculations over a set of independent observations
and different initializations.
We then show the averaged results,
where all quantities are averaged over the numerical experiments.
In all numerical tests,
we discretize the EKI dynamics with an Euler scheme
and comment on the stepsize in the context of each specific example.

\subsection{Linear problem with local observations}
\label{sec:numlin}
We consider the linear model $G(u)=u$,
so that $d_y=d_u$.
Thus, the dimension is the only free parameter
that defines this problem and we can vary the dimension to
study the dimension independence of a LEKI
and contrast it with dimension dependence of the ``vanilla'' EKI.
Since this linear example is characterized by local observations,
we can simply localize with the identity matrix.
Using the identity matrix is the ``optimal'' choice for localizing this problem,
since the model is composed of independent components.
We have run several examples with other localization functions
(e.g., Gaussians) and the results remain qualitatively the same,
as long as the localization radius is chosen to be small
(as is required by this problem).
For the simulations below,
we use a constant time step of $\Delta t = 0.1$
when discretizing the EKI dynamics.

We vary the dimension from five to 100 and, for each $N$, 
perform 200 independent experiments.
In each experiment, we generate an observation by first drawing
a ``true'' $u$ from a Gaussian $\mathcal{N}(0,I)$,
and plugging the result into~(\ref{eq:inv}),
with independent draws from $\eta$ for each experiment.
We apply EKI with ensemble size $J=50$  with and without localization
and record the associated misfits (after 500 iterations) in each experiment.
Figure~\ref{fig:DimIndep}(a) shows the misfits,
averaged over the 200 experiments, associated with EKI and localized EKI
as a function of the dimension.
\begin{figure}[tb]
	\centering
	\includegraphics[width=1\textwidth]{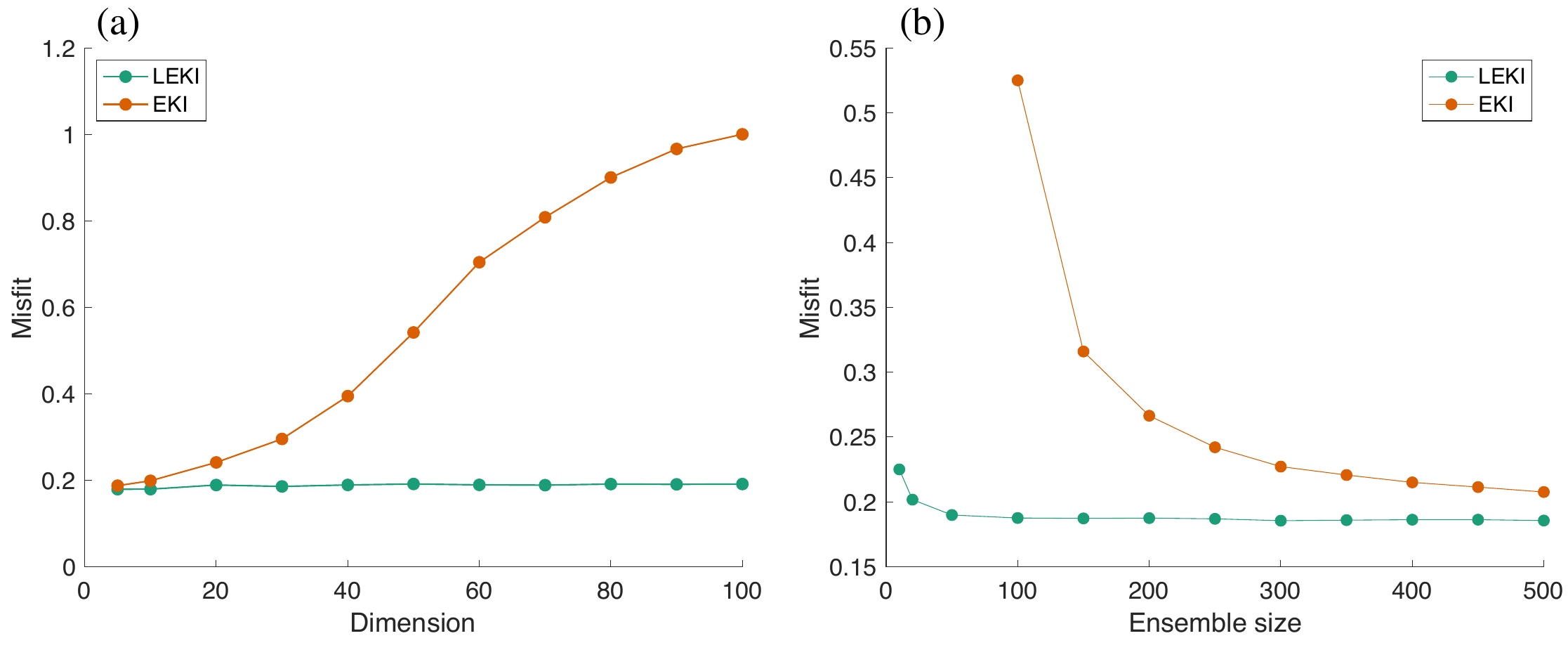}
	\caption{
	Misfits of EKIs for the linear, uncoupled problem described in Section \ref{sec:numlin}.
	(a): Misfit as a function of dimension for EKI (orange) and LEKI (green).
	The ensemble size is $J=50$.
	(b): Misfit as a function of ensemble size for a problem of dimension $100$.}
	\label{fig:DimIndep}
\end{figure}
The experiments indicate that localization helps EKI to converge to a small misfit, 
even if the ensemble size is smaller than the dimension of the problem.
The reason is that the EKI ensemble, after localization, is not confined to the subspace spanned
by the initial ensemble.
This happens because the localized ensemble covariances are full rank,
which is higher than the rank of the unlocalized ensemble covariances (limited by ensemble size).

The benefits of localizing EKI can be illustrated further by numerical experiments
with a fixed dimension, but varying the ensemble size.
Specifically, we fix the dimension to $100$ and vary the ensemble size from ten to 500.
For a given ensemble size, we generate observations,
apply EKI with and without localization,
and compute the associated misfits (after 500 iterations).
Figure~\ref{fig:DimIndep}(b) shows the misfit, averaged over 200 experiments, 
as a function of ensemble size for EKI and localized EKI.
The experiments indicate that localization accelerates the convergence of EKI
with respect to the ensemble size.

In summary,
the two numerical experiments above demonstrate the dimension independence
of the localized EKI: the ensemble size required, e.g., 
to get to a low misfit, 
is independent of the dimension of the problem.
The unlocalized EKI, however, is not dimension independent
because its required ensemble size is a function of the dimension of the problem.
Moreover, localization accelerates the convergence of EKI
with respect to the ensemble size (at any fixed dimension).

We further study the collapse of the localized EKI ensemble in this example.
We now fix the dimension to $d_u=d_y=50$
and consider the rate at which the ensemble collapses
and at which error is reduced during the iteration.
To this extent, we run 100 independent experiments
(independent ground truth, observations and initial ensembles)
and apply EKI with $J=50$ and $J=10^3$ ensemble members,
as well as a LEKI with $J=20$ ensemble members.
We record the trace of the ensemble covariance (tr$(C^{uu})$),
the min.\slash max. values of its diagonal elements ($\|C^{uu}\|_{\max}$,$\|C^{uu}\|_{\min}$),
and the misfit and Max.~Error after each iteration for each of the 100 experiments.
We illustrate the results in Figure~\ref{fig:Collapse}.
\begin{figure}[tb]
	\centering
	\includegraphics[width=1\textwidth]{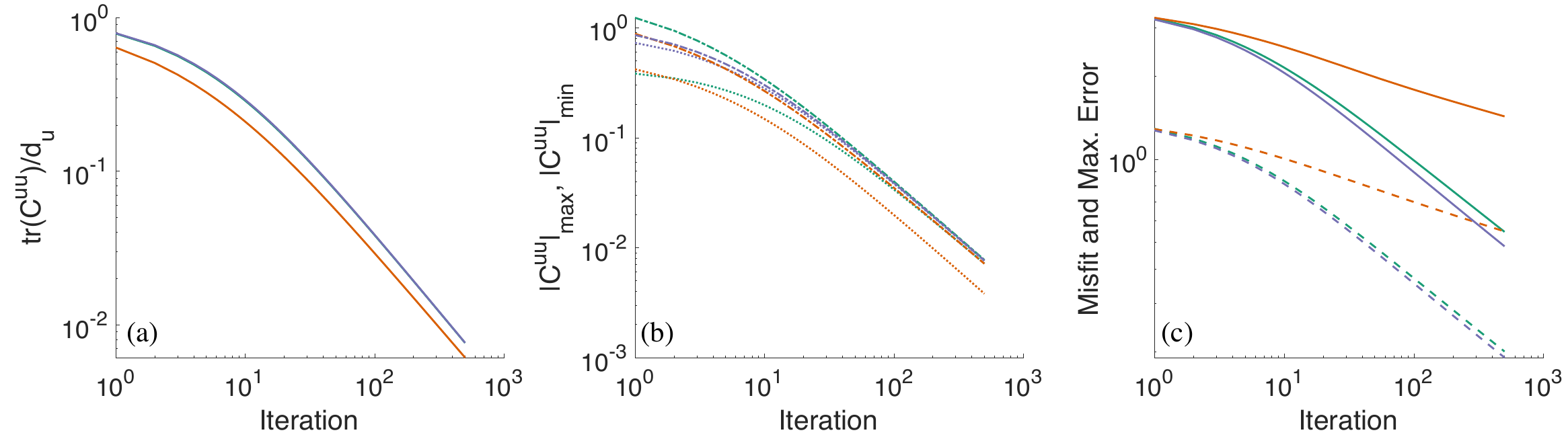}
	\caption{
	Collapse of EKI ensembles for the linear, uncoupled problem described in Section \ref{sec:numlin}.
	Green: LEKI with $J=20$.
	Orange: EKI with $J=50$.
	Purple: EKI with $J=1,000$.
	(a) Average of tr$(C^{uu})$ 
	as a function of the iteration number.
	LEKI (green) not visible because it is nearly identical to the EKI with large ensemble size (purple).
	(b) Average of the largest (dashed) and smallest (dotted) diagonal element of $C^{uu}$ as a function of the iteration number.
	(c) Average of misfit (dashed) and Max.~Error (solid)
	as a function of the iteration number.
	All averages are taken over 100 independent experiments.
	}
	\label{fig:Collapse}
\end{figure}

We conclude from Figures~\ref{fig:Collapse}(a) and (b)
that the ensembles of all EKIs collapse at (similar) fixed rates.
A similar rate of collapse, however, 
does not imply a similar reduction in error in all three EKIs.
Only the EKI with a large ensemble ($J=10^3$)
and the LEKI lead to a significant reduction in error
(see Figure~\ref{fig:Collapse}(c)).
The rate at which error is reduced is comparable for the 
large-ensemble EKI and the localized EKI.
This demonstrates, again, the necessity for localization
if one wants to (or needs to) keep the ensemble size small.

To be sure, we do not mean to use this example to suggest to use LEKI on a trivial optimization problem --
the problem at hand is easy to solve analytically.
Rather, our goal is to illustrate the benefits of localization on the simplest system we could come up with,
and which has been used, e.g., in meteorology, to study and illustrate the collapse of particle filters
\cite{Bickel,Bickel2,Bickel3,Snyder,MHS17,SBM2015}

\subsection{Nonlinear model, non-local observations}
\label{sec:numnonlin}
We now consider a nonlinear model for which $d_u=d_y$
and
\begin{equation}
	y_i = u_i-\sqrt{3}\hat{u}_i^2+\hat{u}_i^3, \quad i=1,\dots,d_u,
\end{equation}
where
\begin{equation}
	\hat{u}_i = \frac{1}{10}\sum_{j=-5}^{5} u_{i-j}, 
\end{equation}
is the average of ten ``neighboring'' elements of $u$. 
We assume zero boundary conditions to compute $\hat{u}_i$
for small and large indices (averaging only the remaining components).
This model $G(u)$ is thus characterized by a nonlinear coupling of its components,
however, the coupling is is confined to small neighborhoods, 
because only a few of the components of $u$ are averaged when computing $\hat{u}$.

We repeat the numerical tests above with this model
and localize the EKI with a centralized localization scheme 
(see \eqref{eq:cenloc}).
The localization function is a Gaussian with a length scale equal to one and the center is $i(j)=j$.
The time discretization of the EKI dynamics is an Euler scheme with a constant time step $\Delta t = 0.05$. 
We perform, for each version of the EKI, 100 iterations and 200 independent experiments.
Figure~\ref{fig:DimIndepNonlin} shows the misfit as a function of dimension,
using a fixed ensemble size of $J=50$.
\begin{figure}[tb]
	\centering
	\includegraphics[width=1\textwidth]{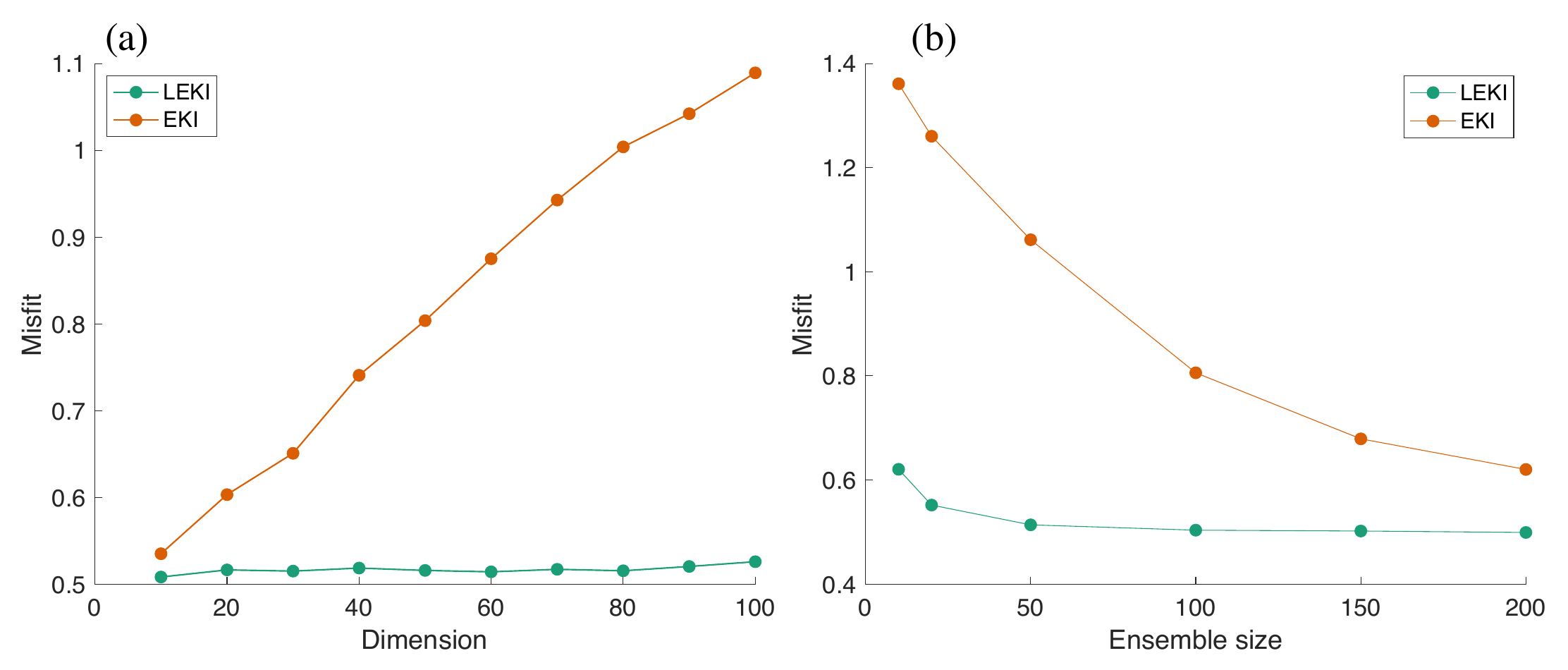}
	\caption{
	Misfit of EKIs for the nonlinear, coupled problem described in Section \ref{sec:numnonlin}.
	(a): Misfit as a function of dimension for EKI (orange) and LEKI (green).
	The ensemble size is $J=50$.
	(b): Misfit as a function of ensemble size for a problem of dimension $50$.}
	\label{fig:DimIndepNonlin}
\end{figure}
As in the linear example,
the misfit of the localized EKI is independent of the dimension,
while the misfit of the ``vanilla'' EKI grows with dimension.
Note that the localization here is not ``optimal'' 
because the localization function we chose is not capable
of reflecting the actual problem structure.
Nonetheless, the localization, even if done ``poorly''
has a tremendous effect on misfit and how quickly it decays during the iteration.

Figure~\ref{fig:CollapseNonlin} illustrates the collapse of the localized and unlocalized EKI ensembles.
\begin{figure}[tb]
	\centering
	\includegraphics[width=1\textwidth]{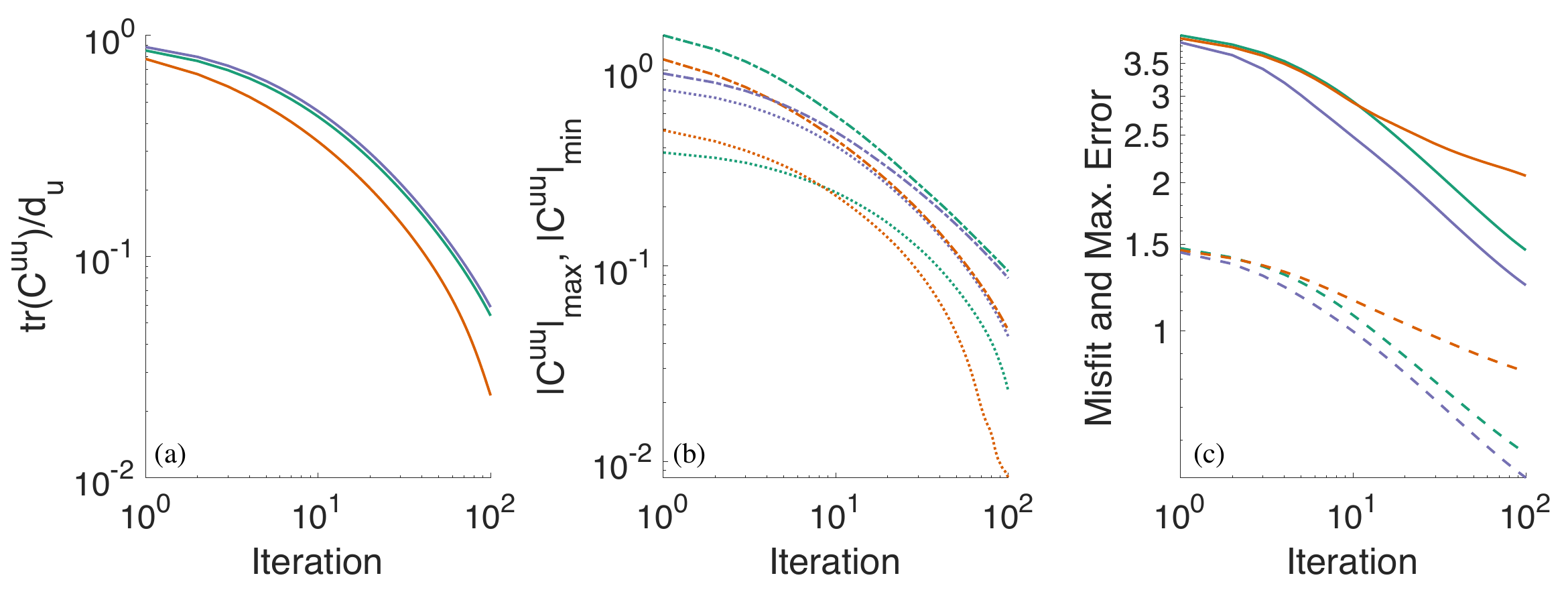}
	\caption{
	Collapse of EKI ensembles for the nonlinear, coupled problem described in Section \ref{sec:numnonlin}. The dimension is set to 50.
	Green: localized EKI with $J=50$.
	Orange: EKI with $J=50$.
	Purple: EKI with $J=10^3$.
	(a) Average of tr$(C^{uu})$ as a function of the iteration number.
	LEKI (green) not visible because it is nearly identical to the EKI with large ensemble size (purple).
	(b) Average of the largest (dashed) and smallest (dotted) diagonal element of $C^{uu}$ as a function of the iteration number.
	(c) Average of misfit (dashed) and Max.~Error (solid)
	as a function of the iteration number.
	LEKI (green) not visible because it is nearly identical to the EKI with large ensemble size (purple).
	All averages are taken over ten independent experiments.
	}
	\label{fig:CollapseNonlin}
\end{figure}
Shown are results averaged over ten independent experiments.
As in the linear example, 
we see that the ensembles of all EKIs collapse at (similar) fixed rates,
but the errors (misfit and Max.~Error)
decay much quicker for the LEKI (or large ensemble EKI),
than for the EKI with a small ensemble.
In summary, we observe in the linear and nonlinear examples
that the localized EKI can achieve a similar performance (in terms of collapse and misfit\slash error decrease)
as an unlocalized EKI with a much larger ensemble.
The reason, as explained above, is that LEKI can break out of the subspace spanned by the initial ensemble,
which makes it equivalent to an unlocalized EKI with a much larger ensemble size.
For these reasons, we also find that LEKI is dimension independent
(error decay or collapse are independent of dimension at fixed ensemble size),
while EKI has a strong dimension dependence (required ensemble size grows with dimension).

\subsection{Lorenz'96}
\label{sec:numl96}
The Lorenz'96 (L96) model \cite{L96} has been used in many studies of data assimilation and ensemble Kalman filtering as a simplified ``toy'' problem to illustrate the need for localization in EnKFs.
We follow this lead and apply TEKI (Tikhonov regularized EKI) to estimate the initial conditions of the L96 model given noisy observations.
The regularization we use is simple -- we regularize with the identity matrix --  because our main goal is to show the effects of localization in TEKI.

In brief, the L96 model is the ordinary differential equation (ODE)
\begin{equation}
	\frac{\text{d}x_k}{\text{d}t} \,=\, -x_k -x_{k-1}(x_{k-2}-x_{k+1}) + F,
\end{equation}
where $k=1,\dots,d_u$, $F = 8$ is a forcing term,
and where we assume a periodic domain so that 
$x_{-1} = x_{d_u-1}$, $x_{0} = x_{d_u}$ and $x_{d_u+1} = x_{1}$. 
The unknown parameter, $u$ is the initial condition $x(0)$.
The observations, $y$, are the state at time $t=0.2$,
perturbed by Gaussian noise with mean zero and covariance matrix equal to the identity matrix.
Note that we observe every state variable to avoid issues with observability.
We discretize the ODE with a simple first-order Euler scheme and time step $0.05$
(again, because we are mostly concerned with showing the effects of localization in TEKI).
We use a time step of 0.1 to discretize the TEKI flow,
using the same numerical scheme as in the previous two examples (also an Euler scheme).
We perform 100 TEKI iterations, starting with an ensemble drawn from a Gaussian distribution with mean zero
and covariance matrix equal the identity matrix.
We then compute the root mean square error (RMSE)
\begin{equation}
	\text{RMSE} = \left( \frac{1}{d_u} \sum_{j=1}^{d_u} (x^t_j - m_j)^2\right)^{1/2},
\end{equation}
where $x^t$ is the true initial condition and $m$ is the mean of the TEKI ensemble.

We perform 100 independent numerical experiments.
For the first one, we integrate the L96 equations, starting from a random state,
for 1000 time units and take the last state of this sequence as the initial condition we invert for.
For the remaining 99 experiments, we integrate the initial condition of the previous experiment for 1000 time units
and set the last state of that sequence to be the initial condition we seek with TEKI.
In this way, we average RMSE over the attractor of the L96 model.
\begin{figure}[tb]
	\centering
	\includegraphics[width=1\textwidth]{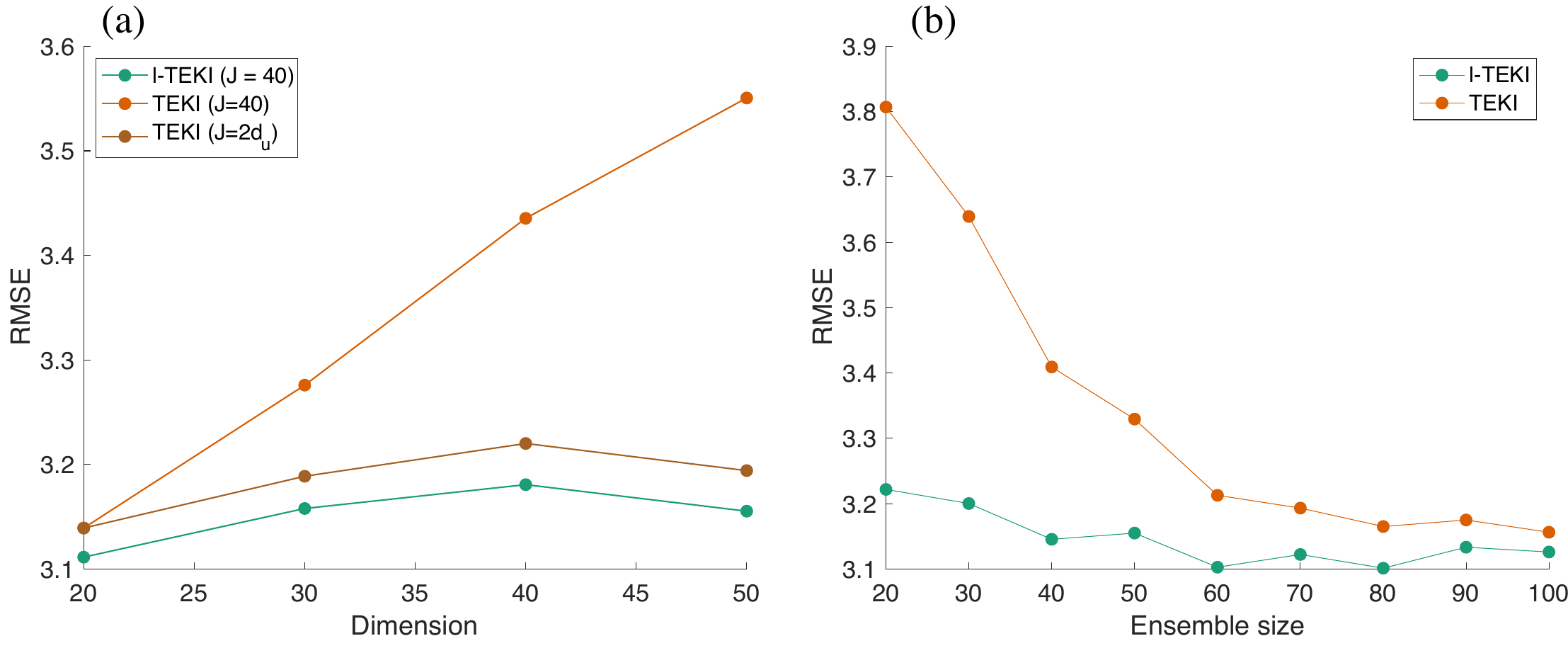}
	\caption{
	RMSE of TEKIs when estimating the initial condition of the L96 model described in Section \ref{sec:numl96}. 
	(a): RMSE as a function of dimension for TEKI with a small ensemble size ($J=40$, orange),
	for TEKI with a large ensemble size ($J=2d_u$, brown), and localized TEKI with a small ensemble size ($J=40$, green).
	(b): RMSE as a function of ensemble size for TEKI (orange) and localized TEKI (green). 
	The dimension of the problem is $d_u=40$.
	}
	\label{fig:L96}
\end{figure}

Results are summarized in Figure~\ref{fig:L96}.
For panel (a), we repeat what we did before and hold the ensemble sized fixed and increase the dimension  of the L96 model from 20 to 50. Shown is the RMSE as a function of the dimension. We note, again, that the RMSE (and misfit, not shown) increases quickly with dimension in case of TEKI (orange), but localization can keep the error (nearly) constant.
We further compare localized TEKI to a TEKI for which we increase the ensemble size, $J$, with dimension at a constant rate ($J=2d_u$). We note that localizing TEKI leads to an RMSE similar to what TEKI can achieve at much larger ensemble sizes.
This is reinforced in panel (b), where we show RMSE as a function of ensemble size for a problem of dimension $d_u=40$.
The RMSE of the localized TEKI with a small ensemble size ($J=30$) is comparable to the RMSE of a TEKI with ensemble size $J=100$.
We have thus demonstrated that localization significantly reduces the required ensemble size in EKI\slash TEKI in three different problems, which reinforces our theory, and further indicates that the theory is robust because the strict assumptions we made to derive the theoretical results are not  always satisfied in the numerical examples we tried.

\subsection{Inversion of DC resistivity field data}
\label{sec:numfield}
Electromagnetic inversions are one of the few tools we have to probe the Earth's crust.
Put simply, one can use electromagnetic inferences to map resistivity
of the Earth, because different types of ``rock'' (partial melt, the mantle, hydrocarbons) 
have different resistivities.
Here we apply localized EKI to invert the ``Schlumberger data set'' 
\cite{Occam}.
These data are DC resistivity field data and can be used to invert for
Earth's resistivity up to a depth of tens of kilometers.

Specifically, the data, $y$, shown as red dots in Figure~\ref{fig:SchlumIllu}(a)
are 29 measurements of apparent resistivity. 
These data are modeled by a 1D layered model of electrical resistivity,
that is described by the number of layers, $d_u$,
the layers' thicknesses, $t_i$ and associated resistivities, $u_i$. 
The details of the model can be found, e.g., in \cite{Occam},
but, in short, data and model are connected by the integral
 \begin{equation}
 	y_{j} = \left(\frac{AB}{2}\right)_j^2 \int_0^\infty T_1(\lambda) \,
	J_1\left(\left(\frac{AB}{2}\right)_j\lambda\right)\,\lambda\,\text{d}\lambda.
 \end{equation}
Here, $j=1,\dots,29$ is an index for the data, $J_1$ is the Bessel function of the first kind,
$(AB/2)_j$ (given) are half-electrode spacings and $T_1(\lambda)$
is the Koefoed resistivity transform,
which, after discretization, can be computed from the recursion
\begin{equation}
	T_i = 
	\frac{T_{i+1}+u_i\,\text{tanh}(\lambda t_i)}{1+T_{i+1}\,\text{tanh}(\lambda t_i)/u_i},
\end{equation}
where $T_i$ is the transform evaluated at the top of the $i$th layer.
The recursion starts with $T_{d_u}=u_{d_u}$, at the top of the terminating half-space.
We define $d_u=20$ layers that are logarithmically spaced between $10^{-1}$~m and $10^{5}$~m.
The unknown parameters we invert for are the layer resistivities $u_i$.

We apply EKI with an ensemble size $J=10$ 
(much smaller than the parameter- or data dimensions).
The initial ensemble is generated by drawing
from a uniform distribution between $0.5$~$\Omega$m and $5$~$\Omega$m,
independently for each layer 
(the upper and lower bounds for the resistivities are chosen based on the descriptions in \cite{Occam}).
We emphasize that the initial ensemble does not fit the data well
because it essentially consists of uncorrelated noise (within reasonable bounds for resistivity).
The EKI starts with a time step of $\Delta t=0.01$,
but the time step is increased, depending on the scaled misfit
at the current iteration.
The scaled misfit is defined by
\begin{equation}
	\text{Misfit}_s= 
	\left(
		\frac{1}{d_y} \sum_{i=1}^{d_y} \left(
		\frac{y_i -\hat{y}_i}{s^i}
		\right)^2
	\right)^{0.5},
\end{equation}
where $s^i$ are standard deviations of the expected errors in the data
(these are specified as part of the Schlumberger data set),
where $y_i$ are the measured apparent resistivities,
and $\hat{y}_i$ is the mean of the EKI ensemble. 
Once the scaled misfit is below eight, we set $\Delta t=0.1$
and once scaled misfit is below 6, we set $\Delta t=0.5$.

To localize the EKI,
we use a centralized Gaussian localization function,
but the spatial variable is in log-space (which is natural for this problem)
and, since $d_y\neq d_u$,
the localization matrix for $C^{up}$ is no longer square.
We chose the length-scale that defines the Gaussian localization
function to be $L=2$.
This is perhaps quite far from an ideal localization,
because we anticipate that the covariance structure 
is not necessarily stationary 
(covariances may extend over larger spatial scales with depth).

We iterate the EKI for a maximum of 2,000 iterations,
but stop the iteration if the scaled misfit is below 1.1;
we also stop the iteration when we encounter unphysical behavior 
within the EKI ensemble (leading to NaNs in the model output).
During the EKI iteration, we occasionally encounter negative resistivities,
which are unphysical, because we do not incorporate any constraints into the EKI.
If an ensemble member exhibits a negative resistivity, 
we set its value to the minimum resistivity of $0.1$~$\Omega$m.

We note that the results one obtains with EKI
vary quite significantly with the initial ensemble.
This is perhaps not surprising because
(\emph{i}) the initial ensemble is essentially composed of noise;
(\emph{ii}) it is known that many models can fit the data equally well
\cite{Occam}, so that, starting from noise, the EKI will find several local minima.
The latter can be addressed by incorporating regularization,
but we do not pursue this here.
Instead, we repeatedly perform EKIs with different initial ensembles
and discuss the results.

Figure~\ref{fig:SchlumIllu} illustrates the results of six localized EKI inversions.
\begin{figure}[tb]
	\centering
	\includegraphics[width=1\textwidth]{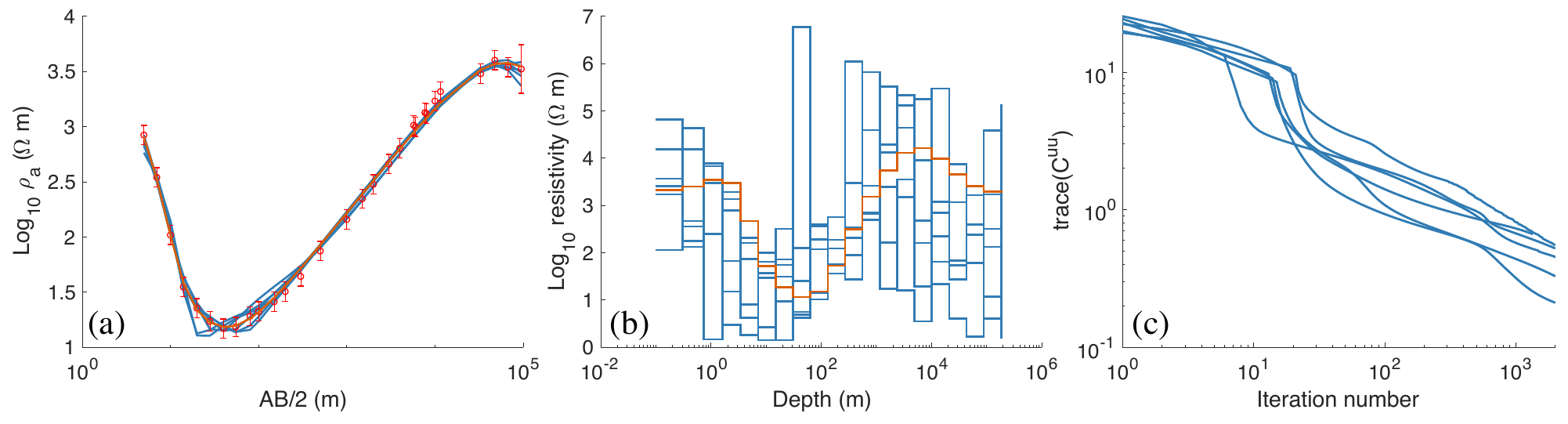}
	\caption{
	Localized EKI on resistivity field data described in Section \ref{sec:numfield}.
	(a) Apparent resistivity data (red error bars) 
	and six LEKI reconstructions after 2000 iterations (blue).
	The EKIs start with different initial conditions.	
	Shown in orange (often hidden) is the result of an Occam inversion (see text for details).
	(b) Resistivity as a function of depth.
	The averages of six LEKI ensembles after 2000 iterations are shown in blue
	and the result of an Occam inversion is shown in orange.
	(c) Trace of the LEKI ensemble covariance as a function of the iteration number
	for six LEKIs, initialized with different initial ensembles.
	}
	\label{fig:SchlumIllu}
\end{figure}
Panel (a) shows the apparent resistivities computed from the model outputs for the means of six LEKI ensembles
after 2000 iterations along with the data.
For comparison, we also performed an Occam inversion \cite{Occam},
in which we compute the layered resistivities using a (gradient-based) optimization.
The Occam inversion, however, makes use of Tikonov regularization,
while the LEKI does not.
Nonetheless, we observe that LEKI discovers models that exhibit a good fit to the data
(to within the assumed errors), which is comparable to the data fit of an alternative technique.
Panel~(b) shows the resistivity models that lead to the apparent resistivities in panel~(a).
We note that the models that LEKI discovers are not necessarily similar,
but lead to a similarly good data fit (and these correspond to local minima of the unregularized loss function). 
The fact that there exist several models that fit the data equally well is a well-known characteristic of DC resistivity problems.
We can thus conclude that LEKI can find (local) minima of the data misfit quite efficiently.
Panel~(c) illustrates the collapse of the LEKI ensembles
and shows the trace of the LEKI ensemble covariance as a function of the iteration.
We note that the collapse occurs at similar rates, 
independently of the initial ensemble.

To demonstrate the beneficial effects of localization of EKI,
we compare the localized EKI to an unlocalized EKI with the same, small ensemble size ($J=10$).
We now perform 50 experiments, with a different initialization of the LEKI and EKI in each experiment
(but the EKI and LEKI start with the same ensemble).
In each experiment, we record the scaled misfit at the end of the iteration
and the exit condition:
target scaled misfit of 1.1 is reached, or number of iterations exceeds 2000, or failure\slash NaN.
We summarize the results of these experiments in Table~\ref{tab:SchlumResults},
where we show statistics of the scaled misfit (computed from all runs that did \emph{not} fail)
and the exit condition.
\begin{table}[tb]
\caption{Summary of results of 50 initializations of EKIs}
\begin{center}
\begin{tabular}{lcccccc}
& \multicolumn{3}{c}{RMSE} & \multicolumn{2}{c}{Exit condition}  \\
& Mean & Median & Std.& Target reached & Failed \\\hline
LEKI& 2.02 & 1.42 & 1.21 & 40 & 10 \\
EKI & 2.94 & 1.72 & 2.80 & 1 & 15
\end{tabular}
\end{center}
\label{tab:SchlumResults}
\end{table}
We note that the scaled misfit after 2000 iterations
is smaller for LEKI than for EKI
(in both mean and median),
and that the standard deviation of scaled misfit is also smaller for LEKI.
A histogram of RMSE of all LEKI\slash EKI is shown in Figure~\ref{fig:SchlumComparison}
to supplement the information from mean, median and standard deviation shown in the table.
\begin{figure}[tb]
	\centering
	\includegraphics[width=0.4\textwidth]{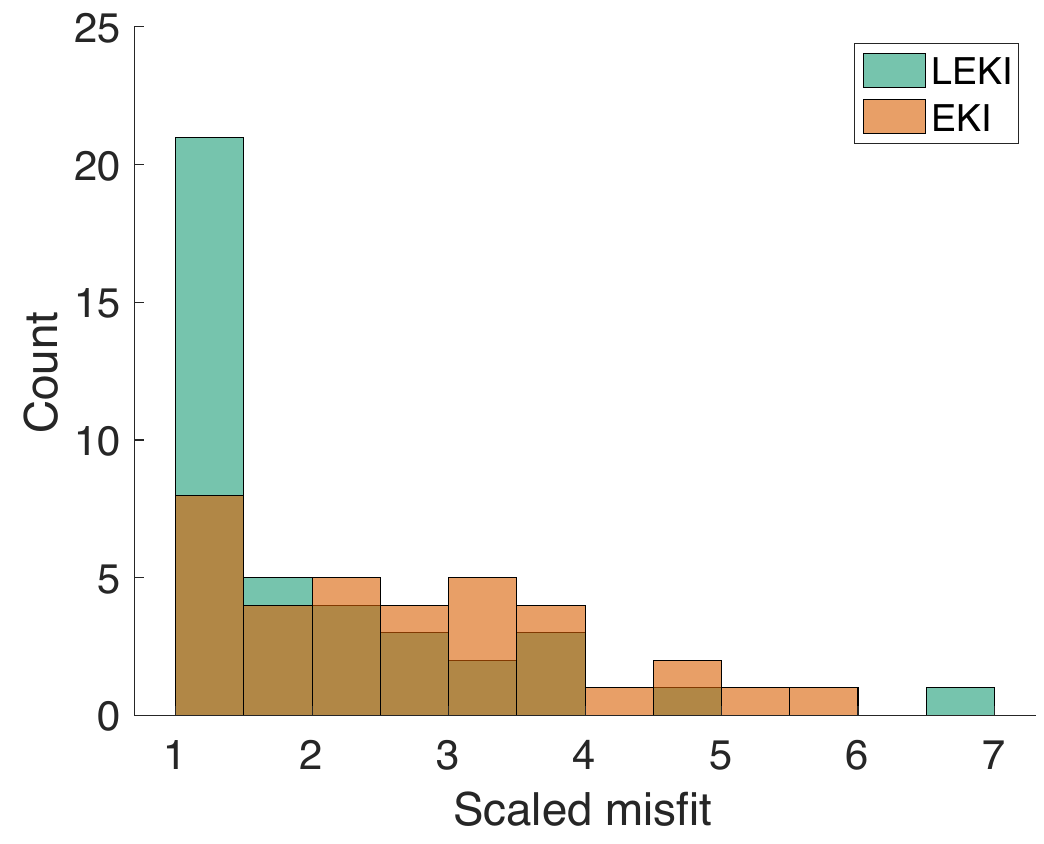}
	\caption{
	Histogram of scaled EKI misfit after 2,000 iterations of LEKI (green) and EKI (orange). The problem is described in Section \ref{sec:numfield}.
	}
	\label{fig:SchlumComparison}
\end{figure}
In particular, we note the large peak around one for the localized EKI. 
Our results thus demonstrate, again, that localization (even if it is not done perfectly)
helps to accelerate the EKI convergence.
Localization also stabilizes the inversion, as can be seen from the fact
that LEKI failed in fewer cases than the unlocalized EKI.

\section{Summary and conclusions}
We described how ideas akin to localization in ensemble data assimilation and ensemble Kalman filtering can be used in ensemble Kalman inversion (EKI). 
In brief, the idea of localization is to enforce an assumed correlation structure on ensemble estimates of covariance matrices within EKI.
We demonstrate, in theory and in practice, that localization brings about significant computational advantages,
the most startling being that localization enables the EKI ensemble to break out of the subspace spanned by the initial ensemble.
This subspace property of EKI implies that EKI requires an ensemble size (at least) proportional to the dimension of the problem (number of unknown parameters) -- this is impractical in most relevant problems. 
Localization does away with this requirement and enables a dimension independent application of the localized EKI (LEKI).
Specifically, we demonstrate in several examples, that the required ensemble size is independent of the dimension of the problem, as long as localization can be applied.
We formalized, for the first time,  the effects of localization on EKI and proved theorems on ensemble collapse and convergence rates. 
While some of our theoretical results require relatively strict assumptions, which may be hard to validate in practice, our numerical experiments indicate that LEKI can work well, even if some of our assumptions may only be partly satisfied.
This suggests that future work can significantly tighten the bounds we derived and relax assumptions.
Our work is a first meaningful step towards placing the largely empirical approach of localization on a mathematically sound footing within EKI, showing that localization is a required step to make EKI applicable to large-scale problems.

\section*{Acknowledgements}
We thank Andrew Stuart (Caltech) for interesting discussion of localization in EKI.
XT is supported by the Singapore Ministry of Education (MOE) grant R-146-000-292-114. MM is supported by the US Office of Naval Research (ONR) grant N00014-21-1-2309.


\appendix
\section{Technical Lemmas}
We present some technical estimates here. Many of them or similar variants can be found in the literature. 

\begin{lem}
\label{lem:norm}
For any $N\times N$ matrix $A$, the following holds
\begin{align}
&\|A\|_{\max} \le \|A\|,\label{eq:maxle2norm}\\
&\|A\| \le \sqrt{\|A\|_1 \|A^\top \|_1}\label{eq:2normlestuff}.
\end{align}
\end{lem}
\begin{proof}
Inequality (\ref{eq:maxle2norm}) follows via
\[
\|A\|_{\max}=\max_{i,j}|[A]_{i,j}|=\max_{i,j}|[e_t]_i^\top A e_j|\le \|A\|,
\]
where $[e_t]_i$ and $e_j$ are the $i$-th and $j$-th standard Euclidean basis vector.
Inequality \eqref{eq:2normlestuff} follows from \cite{MTM19} Lemma B.2.
\end{proof}

\begin{lem}
\label{lem:boundonderivative}
Suppose $X_t=[x_1(t),\ldots, x_n(t)]$ jointly follows an ODE, $\frac{d}{dt}X_t=F(X_t)$. Let $m_t=\max_{1\leq i\leq n}\{x_i(t)\}$. Let $i_t$ be the smallest index $i$ such that $x_i(t)=m_t$. Suppose there is a continuous function $g(x,t)$ such that for any $t\geq 0$, 
\[
\frac{d}{dt}x_{i_t}(t)\leq g(x_{i_t}(t),t).
\]

Suppose $y_t$ satisfies $\frac{d}{dt}y_t=g(y_t,t)+\delta_0$ for a fixed $\delta_0>0$ and $y_0>m_0$, then for all $t>0$,  $y_t>m_t$.
\end{lem}
\begin{proof}
Let $t_1=\inf\{t>0, y_t\leq m_t\}$. By continuity of $m_t$ and $y_t$, $t_1>0$. Suppose $t_1$ is finite, then $y_{t_1}=m_{t_1}$.  Therefore 
\[
\frac{d}{dt}x_{i_{t_1}}(t_1)\leq g(x_{i_{t_1}}(t_1),t_1)=g(y_{t_1},t_1)=\frac{d}{dt} y(t_1)-\delta_0.
\]
This indicates that for sufficiently small $\delta>0$,
\[
x_{i_{t_1}}(t_1-\delta)> x_{i_{t_1}}(t_1)-\delta g(x_{i_{t_1}}(t),t_1)-\frac12  \delta\delta_0>y(t_1)-\delta g(y(t_1),t_1)+\frac12 \delta\delta_0>y(t_{1}-\delta). 
\]
This contradicts with the definition of $t_1$. Therefore $t_1=\infty$. 
\end{proof}

\begin{lem}
\label{lem:ric}
Suppose $a>0$ and $\sigma\geq 0, $ the solution to the Riccati equation
\[
\dot{y}=-ay^2-\frac{b}{t+1}y+\frac{\sigma}{(t+1)^2},\quad ,
\]
is given by 
\[
y_t=\frac{c_-(t+1)^{-c_-}+Bc_+(t+1)^{-c_+}}{-a[(t+1)^{1-c_-}+B(t+1)^{1-c_+}]}
\]
where $c_-<c_+$ are the roots to the equation
\[
c^2+c-bc=a\sigma,
\]
and $B$ is a constant so that the initial condition holds. In particular 
\begin{enumerate}[1)]
\item If $\sigma>0$,  $c_-<0<c_+$, so as $t\to \infty$, $y_t\to \frac{c_-}{-a(t+1)}$. And for any $\delta>0$, there is a $t_0$ so that if $t\geq t_0$, $y_t\leq \frac{c_-}{-a(t+1)(1-\delta)}$. 
\item If $\sigma=0,b=0$, $c_-=-1,c_+=0$. 
\end{enumerate}
\end{lem}
\begin{proof}
We verify that the ODE holds with our solution
\begin{align*}
\dot{y}_t=&\frac{(c^2_-(t+1)^{-c_--1}+Bc^2_+(t+1)^{-c_+-1})((t+1)^{1-c_-}+B(t+1)^{1-c_+})}{a[(t+1)^{1-c_-}+B(t+1)^{1-c_+}]^2}\\
&+\frac{(c_-(t+1)^{-c_-}+Bc_+(t+1)^{-c_+})((1-c_-)(t+1)^{-c_-}+(1-c_+)B(t+1)^{-c_+})}{a[(t+1)^{1-c_-}+B(t+1)^{1-c_+}]^2}\\
&=\frac{c_-(t+1)^{-2c_-}+B((c_--c_+)^2+c_-+c_+)(t+1)^{-c_--c_+}+B^2c^2_+(t+1)^{-2c_+}}{a[(t+1)^{1-c_-}+B(t+1)^{1-c_+}]^2}
\end{align*}
\begin{align*}
ay^2_t=&\frac{c^2_-(t+1)^{-2c_-}+B^2c^2_+(t+1)^{-2c_+}+2Bc_-c_+(t+1)^{-c_--c_+}}{a[(t+1)^{1-c_-}+B(t+1)^{1-c_+}]^2}
\end{align*}
\[
\frac{by_t}{t+1}=\frac{-bc_-(t+1)^{-2c_-}-B^2bc_+(t+1)^{-2c_+}-(c_-+c_+)Bb (t+1)^{-c_--c_+}}{a[(t+1)^{1-c_-}+B(t+1)^{1-c_+}]}
\]
Therefore 
\begin{align*}
\dot{y}_t+ay_t^2=&\frac{a\sigma  [(t+1)^{-2c_-}+2B(t+1)^{-c_--c_+}+B^2(t+1)^{-2c_+}]}{a[(t+1)^{1-c_-}+B(t+1)^{1-c_+}]^2}=\frac{a\sigma}{(t+1)^2}. 
\end{align*}

\end{proof}

\begin{lem}
\label{lem:Gronwall}
Suppose the following holds 
\[
\frac{d}{dt}x_t\leq -a_t x_t+b_t,
\]
where 
\[
a_t\geq 1_{t\geq t_0}\frac{\alpha}{1+t}-1_{t< t_0}\beta,\quad b_t\leq \frac{M}{(1+t)^{1+\gamma}},\quad t\geq t_0
\]
Then 
\[
x_t\lesssim \frac{x_0}{(t+1)^{\alpha}}+\frac{M(\log (t+1))^{1_{\gamma=\alpha}}}{(t+1)^{\min\{\gamma,\alpha\}}}.
\]
\end{lem}
\begin{proof}
By Gronwall's inequality 
\[
x_t\leq \exp(-\int^t_s a_r dr) x_s +\int^t_s \exp(-\int^t_u a_r dr) b_udu
\]
Apply this with $t=t_0, s=0$, we find that 
\[
x_{t_0}\leq  e^{\beta t_0} x_0+\int^{t_0}_0 e^{\beta (t_0-s)}M ds\leq e^{\beta t_0}(x_0+M/\beta).
\]
Apply the same formula to $t=t, s=t_0$, note that 
\[
\exp(-\int^t_u a_r dr)\leq \exp(-\alpha\log \tfrac{t+1}{u+1})=\left(\frac{u+1}{t+1}\right)^{\alpha}
\]
we find
\[
x_t\leq \left(\frac{t_0+1}{t+1}\right)^{\alpha}x_{t_0} +M\int^t_{t_0} \frac{(u+1)^{\alpha-1-\gamma}}{(t+1)^{\alpha}}  du
\]
When $\gamma>\alpha$, we find that 
\[
x_t\leq \left(\frac{t_0+1}{t+1}\right)^{\alpha}x_{t_0} +\frac{M}{(\gamma-\alpha)(t+1)^{\alpha}(t_0+1)^{\gamma-\alpha}} \lesssim  \frac{x_0+M}{(t+1)^{\alpha}}.
\]
When $\gamma<\alpha$, we find that 
\[
x_t\leq \left(\frac{t_0+1}{t+1}\right)^{\alpha}x_{t_0} +\frac{M(t+1)^{\alpha-\gamma}}{(\alpha-\gamma)(t+1)^{\alpha}} \lesssim  \frac{x_0+M}{(t+1)^{\gamma}}.
\]

\end{proof}
The next argument can also be found in \cite{de2020analysis}
\begin{lem}
\label{lem:vexist}
Suppose we have a symmetric matrix $\phi\in \reals^{d\times d}$ such that 
\begin{enumerate}
\item $\phi_{i,j}\geq 0$ for all $i,j$,
\item $\phi_{i,i}=0$ for all $i$,
\item there is a $\phi_0>0$ such that $\phi_0\leq 1-\sum_{j=1}^d\phi_{i,j}$.
\end{enumerate}
Then if we let $T$ be a random variable of geometric-$\phi_0$ distribution, that is 
\[
P(T=n)=(1-\phi_0)\phi_0^{n-1},\quad n=1,2,\ldots.
\]
Consider a Markov chain $X_t$ on the points $\{1,\ldots, d\}$. Its transition probability is given by 
\[
P(X_{t+1}=j|X_t=i)=\begin{cases}\frac1{1-\phi_0}\phi_{i,j}\quad &j\neq i\\
1-\frac1{1-\phi_0}\sum_{j\neq i}\phi_{i,j}\quad &j=i.
\end{cases}
\]
Fix an index $i\in \{1,\ldots, d\}$. Define a vector $v^{i}$, where its components are given by 
\[
v^{i}_j=\E\left(\sum^T _{k=1} \mathbf{1}_{X_k=i}\bigg|X_1=j\right). 
\]
Then $v^{i}$ satisfies the following properties
\begin{enumerate}[1)]
\item $v^{i}_j\geq 0, \forall j$ and in specific $v^i_i\geq \phi_0$.
\item For all index $j$, $\sum_{l\neq j} \phi_{j,l}v^i_l\leq v^i_j$.
\item $\sum_{j=1}^{d} v^i_j\leq 1$. 
 \end{enumerate}
\end{lem}
\begin{proof}
Since $\sum^T _{k=1} \mathbf{1}_{X_k=i}\geq 0$ a.s., so $v^{i}_j\geq 0$. This also leads to claim 1)
\[
v^i_i=\E\left(\sum^T _{k=1} \mathbf{1}_{X_k=i}\bigg|X_1=i\right)
\geq \E\left(\mathbf{1}_{T=1, X_1=i}\bigg|X_1=i\right)=\phi_0. 
\]
Next, by doing a first step analysis of the Markov chain, we find that 
\begin{equation}
\label{tmp:vij}
v^i_j =\phi_0\cdot \mathbf{1}_{j=i}+(1-\phi_0) \left(1-\frac1{1-\phi_0}\sum_{l\neq j}\phi_{j,l}\right)v^i_j+(1-\phi_0)\cdot \frac1{1-\phi_0}\sum_{l\neq j}\phi_{j,l}v^i_l.
\end{equation}
Since $\sum_{l\neq j}\phi_{j,l}\leq q<1$, we have claim 2) by
\[
v^i_j\geq (1-\phi_0)\cdot \frac1{1-\phi_0}\sum_{l\neq j}\phi_{j,l}v^i_l=\sum_{l\neq j}\phi_{l,j}v^i_l.
\]
Finally we sum \eqref{tmp:vij} over all $j$ and obtain
\begin{align*}
\sum_{j=1}^{d}v^i_j &=\phi_0+(1-\phi_0)\sum_{j=1}^{d}\left(1-\frac1{1-\phi_0}\sum_{l\neq j}\phi_{j,l}\right)v^i_j+\sum_{j=1}^{d}\sum_{l\neq j}\phi_{j,l}v^i_l\\
&\leq \phi_0+\sum_{j=1}^{d}\sum_{l\neq j}\phi_{j,l}v^i_l=\phi_0+ \sum_{l=1}^{d} v_l^i\left(\sum_{j\neq l }\phi_{j,l}\right).
\end{align*}
Therefore we have 
\[
\phi_0\sum_{j=1}^{d}v^i_j \leq \sum_{j=1}^{d}(1-\sum_{j\neq l }\phi_{j,l})v^i_j\leq \phi_0,
\]
which leads to our claim 3). 
\end{proof}

%
%
\bibliographystyle{unsrt}
\bibliography{References}

\end{document}